\documentclass[10pt, en]{article}
\usepackage[margin=1.5cm]{geometry}                		
\DeclareMathSizes{10}{9}{7}{5}

\usepackage[table]{xcolor} 		

\usepackage{graphicx}	
\usepackage{authblk}
\usepackage{abstract}

\usepackage{amssymb, amsthm, bbold, amsmath} 
\usepackage{booktabs}
\usepackage{enumitem}

\usepackage[font=large]{subcaption}
\usepackage{colortbl}

\newcommand{\bx}{\bold{x}}
\newcommand{\bdx}{\delta{\bold{x}}}
\newcommand{\bL}{\bold{L}}

\newtheorem{theorem}{Theorem}
\newtheorem{lemma}{Lemma}
\newtheorem{cor}{Corollary}

\title{Spectral estimates for saddle point matrices arising in weak constraint four-dimensional variational data assimilation}

\author[1]{\normalsize Ieva Dau\v{z}ickait\.{e}*}

\author[1,2]{\normalsize Amos S. Lawless}

\author[1,3]{\normalsize Jennifer A. Scott}

\author[1,2,4]{\normalsize Peter Jan van Leeuwen}

\affil[1]{\footnotesize School Of Mathematical, Physical and Computational Sciences, University of Reading,UK}

\affil[2]{ \footnotesize National Centre for Earth Observation, Reading, UK}

\affil[3]{\footnotesize Scientific Computing Department, STFC Rutherford Appleton Laboratory, UK}

\affil[4]{\footnotesize Department of Atmospheric Science, Colorado State University, USA}

\date{}

\begin{document}

\maketitle

\begin{abstract}
We consider the large-sparse symmetric linear systems of equations that arise in the solution of weak constraint four-dimensional variational data assimilation, a method of high interest for numerical weather prediction. These systems can be written as saddle point systems with a $3 \times 3$ block structure but block eliminations can be performed to reduce them to saddle point systems with a $2 \times 2$ block structure, or further to symmetric positive definite systems. In this paper, we analyse how sensitive the spectra of these matrices are to the number of observations of the underlying dynamical system. We also obtain bounds on the eigenvalues of the matrices. Numerical experiments are used to confirm the theoretical analysis and bounds.
\vspace{\baselineskip}

\noindent {\bf Keywords:} data assimilation, saddle point systems, spectral estimates, weak constraint 4D-Var, sparse linear systems.
\end{abstract}

\let\thefootnote\relax\footnotetext{*Correspondence: Ieva Dau\v{z}ickait\.{e}, Department of Mathematics and Statistics, University of Reading, PO Box 220, Reading RG6 6AX, UK. Email: i.dauzickaite@pgr.reading.ac.uk}

\section{Introduction}\label{sec1}
Data assimilation estimates the state of a dynamical system by combining observations of the system with a prior estimate. The latter is called a background state and it is usually an output of a numerical model that simulates the dynamics of the system. The impact that the observations and the background state have on the state estimate depends on their errors whose statistical properties we assume are known. Data assimilation is used to produce initial conditions in numerical weather prediction (NWP) \cite{Kalnay2002,Swinbank2010}, as well as other areas, e.g. flood forecasting \cite{CHEN2013}, research into atmospheric composition \cite{Elbern1997}, and neuroscience \cite{Moye2018}. In operational applications, the process is made more challenging by the size of the system, e.g. the numerical model may be operating on $10^8$ state variables and $10^5 - 10^6$ observations may be incorporated\cite{Nichols2010,Lawless2013}. Moreover, there is usually a constraint on the time that can be spent on calculations. 

The solution, called the analysis, is obtained by combining the observations and the background state in an optimal way. One approach is to solve a weighted least-squares problem, which requires minimising a cost function. An active research topic in this area is the weak constraint four-dimensional variational (4D-Var) data assimilation method\cite{Tremolet06,Tremolet07,AES_thesis,Bonavita2017,Fisher17,Gratton2018,Freitag2018}. It is employed in the search for states of the system over a time period, called the assimilation window. This method uses a cost function that is formulated under the assumption that the numerical model is not perfect and penalises the weighted discrepancy between the analysis and the observations, the analysis and the background state, and the difference between the analysis and the trajectory given by integrating the dynamical model.

Effective minimisation techniques need evaluations of the cost function and its gradient that involve expensive operations with the dynamical model and its linearised variant. Such approaches are impractical in operational applications.One way to approximate the minimum of the weak constraint 4D-Var is to use an inexact Gauss-Newton method \cite{Gratton2007}, in which a series of linearised quadratic cost functions with a low resolution model are minimised \cite{Courtier1994}, and the minima are used to update the high resolution state estimate. The state estimate update is found by solving sparse symmetric linear systems using an iterative method \cite{SaadBook}.

To increase the potential of using parallel computations when computing the state update with weak constraint 4D-Var, Fisher and G{\"u}rol \cite{Fisher17} introduced a symmetric $3 \times 3$ block saddle point formulation. These resulting large symmetric linear systems are solved using Krylov subspace solvers\cite{Freitag2018,SaadBook,Benzi2005}. One criteria that affects their convergence is the spectra of the coefficient matrices\cite{Benzi2005}. We derive bounds for the eigenvalues of the $3 \times 3$ block matrix using the work of Rusten and Winther\cite{RustenWinther1992}. Also, inspired by the practice in solving saddle point systems that arise from interior point methods \cite{Greif2014,Morini2017}, we reduce the $3\times 3$ block system to a $2 \times 2$ block saddle point formulation and derive eigenvalue bounds for this system. We also consider a $1 \times 1$ block formulation with a positive definite coefficient matrix, which corresponds to the standard method\cite{Tremolet06,Tremolet07}. Some of the blocks in the $3 \times 3$ and $2 \times 2$ block saddle point coefficient matrices, and the $1 \times 1$ block positive definite coefficient matrix depend on the available observations of the dynamical system. We present a novel examination of how adding new observations influences the spectrum of these coefficient matrices. 

In Section \ref{sec:varDA}, we formulate the data assimilation problem and introduce weak constraint 4D-Var with the $3\times 3$ block and $2 \times 2$ block saddle point formulations and the $1 \times 1$ block symmetric positive definite formulation. Eigenvalue bounds for the saddle point and positive definite matrices and results on how the extreme eigenvalues and the bounds depend on the number of observations are presented in Section \ref{sec:eigenvalues_theory}. Section \ref{sec:numerics} illustrates the theoretical considerations using numerical examples, and concluding remarks and future directions are presented in Section \ref{sec:conclusions}.

\section{Variational Data Assimilation}\label{sec:varDA}
The state of the dynamical system of interest at times $t_0 < t_1 < ... <t_N$ is represented by the state vectors $x_0, x_1, \dots, x_N$ with $x_i \in \mathbb{R}^n$. A nonlinear model $m_i$ that is assumed to have errors describes the transition from the state at time $t_i$ to the state at time $t_{i+1}$, i.e.
\begin{equation}\label{eq:state_evolution}
x_{i+1} = m_i (x_i) + \eta_{i+1},
\end{equation}
where $\eta_i$ represents the model error at time $t_i$. It is further assumed that the model errors are Gaussian with zero mean and covariance matrix $Q_i \in \mathbb{R}^{n \times n}$, and that they are uncorrelated in time, i.e. there is no relationship between the model errors at different times. In NWP, the model comes from the discretization of the partial differential equations that describe the flow and thermodynamics of a stratified multiphase fluid in interaction with radiation\cite{Kalnay2002}. It also involves parameters that characterize processes arising at spatial scales that are smaller than the distance between the grid points\cite{Rood2010}. Errors due to the discretization of the equations, errors in the boundary conditions, inaccurate parameters etc. are components of the model error\cite{Griffith2000}.

The background information about the state at time $t_0$ is denoted by $x^b \in \mathbb{R}^n$. $x^b$ usually comes from a previous short range forecast and is chosen to be the first guess of the state. It is assumed that the background term has errors that are Gaussian with zero mean and covariance matrix $B \in \mathbb{R}^{n \times n}$. 

Observations of the dynamical system at time $t_i$ are given by $y_i  \in \mathbb{R}^{p_i}$. In NWP, there are considerably fewer observations than state variables, i.e. $p_i << n$. Also, there may be indirect observations of the variables in the state vector and a comparison is obtained by mapping the state variables to the observation space using a nonlinear operator $\mathcal{H}_i$. For example, satellite observations used in NWP provide top of the atmosphere radiance data, whereas the model operates on different meteorological variables, e.g. temperature, pressure, wind etc. \cite{Andersson2010}
Hence, values of meteorological variables are transformed into top of the atmosphere radiances in order to compare the model output with the observations. In this case, the operator $\mathcal{H}_i$ is nonlinear and complicated. Approximations made when mapping the state variables to the observation space,
different spatial and temporal scales between the model and some observations (observations may give information at a finer resolution than the model), and pre-processing, or quality control, of the observations (see, e.g. Section 5.8 of Kalnay\cite{Kalnay2002}) comprise the representativity error\cite{Janjic2018}. The observation error is made up of the representativity error combined with the error arising due to the limited precision of the measurements. It is assumed to be Gaussian with zero mean and covariance matrix $R_i \in \mathbb{R}^{p_i \times p_i}$. The observation errors are assumed to be uncorrelated in time \cite{Lawless2013}.

\subsection{Weak constraint 4D-Var}
In weak constraint 4D-Var, the analysis $x^a_0, x^a_1, \dots, x^a_N$ is obtained by minimising the following nonlinear cost function
\begin{align}
J(x_0, x_1, \dots, x_N) & =  \frac{1}{2} (x_0 - x^b)^T B^{-1} (x_0 - x^b) + \frac{1}{2} \sum_{i=0}^{N} (y_i - \mathcal{H}_i (x_i))^T R_i^{-1} (y_i - \mathcal{H}_i (x_i)) \label{eq:4D-var_x} \\ \nonumber
& + \frac{1}{2} \sum_{i=0}^{N-1} (x_{i+1} - m_i (x_i))^T Q_{i+1}^{-1} (x_{i+1} - m_i (x_i)).
\end{align}
This cost function is referred to as the state control variable formulation. Here the control variables are defined as the variables with respect to which the cost function is minimised, i.e. $x_0, x_1, \dots, x_N$ in \eqref{eq:4D-var_x}. Choosing different control variables and obtaining different formulations of the cost function is possible \cite{Tremolet06}. If the model is assumed to have no errors (i.e. $x_{i+1} = m_i (x_i)$), the cost function simplifies as the last term in \eqref{eq:4D-var_x} is removed; this is called strong constraint 4D-Var. Rejecting this perfect model assumption and using weak constraint 4D-Var may lead to a better analysis \cite{Tremolet07}. 

Iterative gradient-based optimisation methods are used in practical data assimilation \cite{Talagrand2010,Lawless2013}. These require the cost function and its gradient to be evaluated at every iteration. In operational applications, integrating the model over the assimilation window to evaluate the cost function is computationally expensive. The gradient is obtained by the adjoint method (see, e.g., Section 2 of Lawless\cite{Lawless2013} and Section 3.2 of Talagrand\cite{Talagrand2010} for an introduction), which is a few times more computationally expensive than evaluating the cost function. This makes the minimisation of the nonlinear weak-constraint 4D-Var cost function impractical. Hence, approximations have to be made. We introduce such an approach in the next section.

\subsection{Incremental formulation}\label{sec:incremental}
Minimisation of the 4D-Var cost function in an operational setting is made feasible by employing an iterative Gauss-Newton method, as first proposed by Courtier et al. \cite{Courtier1994} for the strong constraint 4D-Var. In this approach, the solution of the weak constraint 4D-Var is approximated by solving a sequence of linearised problems, i.e. the $l$-th approximation of the state is 
\begin{equation}\label{eq:state_approx}
x_i^{(l+1)} = x_i^{(l)} + \delta x_i^{(l)},\quad i \in \{0,1,\dots,N\},
\end{equation}
where $\delta x_i^{(l)}$ is obtained as the minimiser of the linearised cost function 
\begin{align}\label{eq:incr_state}
J^{\delta} (\delta x_0^{(l)}, \delta x_1^{(l)}, \dots, \delta x_N^{(l)}) & =  (\delta x_0^{(l)} - b^{(l)})^T B^{-1} (\delta x_0^{(l)} - b^{(l)}) \\ \nonumber
& + \frac{1}{2} \sum_{i=0}^{N} (H_i^{(l)} \delta x_i^{(l)} - d_i^{(l)})^T R_i^{-1} (H_i^{(l)} \delta x_i^{(l)} - d_i^{(l)}) \\ \nonumber
& + \frac{1}{2} \sum_{i=0}^{N-1} (M_i^{(l)} \delta x_i^{(l)} - \delta x_{i+1}^{(l)} - \eta_{i+1}^{(l)})^T Q_{i+1}^{-1} (M_i^{(l)} \delta x_i^{(l)} - \delta x_{i+1}^{(l)} - \eta_{i+1}^{(l)}), 
\end{align}
where $b^{(l)} = x_0 ^{(l)}- x^b$, $d_i^{(l)} = y_i - \mathcal{H}_i (x_i^{(l)})$,  $\eta_i^{(l)} = x_i^{(l)} - m_{i-1} (x_{i-1}^{(l)})$ (as in \eqref{eq:state_evolution}) and $M_i^{(l)}$ and $H_i^{(l)}$ are the model $m_i$ and the observation operator $\mathcal{H}_i$, respectively, linearised at $x^{(l)}_i$. Minimisation of \eqref{eq:incr_state} is called the inner loop. The $l$-th outer loop consists of updating the approximation of the state \eqref{eq:state_approx}, linearising the model $m_i$ and the observation operator $\mathcal{H}_i$, and computing the values $b^{(l)}$, $d_i^{(l)}$ and $\eta_i^{(l)}$. This cost function is quadratic, which allows the use of effective minimisation techniques, such as conjugate gradients (see Chapter 5 of Nocedal and Wright\cite{Nocedal06}). In NWP, the computational cost of minimising the 4D-Var cost function is reduced by using a version of the inner loop cost function that is defined for a model with lower spatial resolution, i.e. on a coarser grid \cite{Fisher98}. We do not consider such an approach in the subsequent work, because our results on the change of the spectra of the coefficient matrices and the bounds (that are introduced in the following section) hold for models with any spatial resolution.

For ease of notation, we introduce the following four-dimensional (in the sense that they contain information in space and time) vectors and matrices. These vectors and matrices are indicated in bold.

\[ \bx^{(l)} = \left( \begin{array}{c}
x_0^{(l)} \\
x_1^{(l)} \\
\vdots \\
x_N^{(l)}
\end{array} \right), \ 
\delta \bx^{(l)} = \left( \begin{array}{c}
\delta x_0^{(l)} \\
\delta x_1^{(l)} \\
\vdots \\
\delta x_N^{(l)}
\end{array} \right),
 \bold{b}^{(l)} = \left( \begin{array}{c}
b^{(l)}\\
- \eta_1^{(l)} \\
\vdots \\
-\eta_N^{(l)}
\end{array} \right), \ 
\bold{d}^{(l)} = \left( \begin{array}{c}
y_0  - \mathcal{H}_0 (x_0^{(l)})\\
y_1  - \mathcal{H}_1 (x_1^{(l)}) \\
\vdots \\
y_N  - \mathcal{H}_N (x_N^{(l)}) 
\end{array} \right),
\]
where $\bx^{(l)}, \delta \bx^{(l)}, \bold{b}^{(l)} \in \mathbb{R}^{(N+1)n}$ and $\bold{d}^{(l)}  \in \mathbb{R}^p, p = \Sigma_{i=0}^{N} p_i $. We also define the matrices
\[ \bL^{(l)} = \left( \begin{array}{ccccc}
I         &             &            &                    & \\
-M_0^{(l)} &     I       &            &                    &  \\
          &   -M_1^{(l)}&       I     &                    &  \\
         &             & \ddots & \ddots        &  \\
         &            &              & -M_{N-1}^{(l)}   & I  \\
\end{array} \right),
\qquad
\bold{H}^{(l)} = \left( \begin{array}{cccc}
H_0^{(l)}                &             &              & \\
          &     H_1^{(l)}              &              &  \\
          &            &             \ddots  &  \\
          &            &            &               H_N^{(l)} \\
  \end{array} \right),
\]
where $I\in \mathbb{R}^{n \times n}$ is the identity matrix, $\bL^{(l)} \in \mathbb{R}^{(N+1)n \times (N+1)n}$ and $\bold{H}^{(l)} \in \mathbb{R}^{p \times (N+1)n }$.  
We define the block diagonal covariance matrices 
\[ \bold{D} = \left( \begin{array}{cccc}
B                  &             &              & \\
          &     Q_1              &              &  \\

          &            &             \ddots  &  \\
          &            &            &               Q_N  \\
\end{array} \right) \quad \text{and} \quad
\bold{R} = \left( \begin{array}{cccc}
R_0                   &             &              & \\
          &     R_1              &              &  \\
          &            &             \ddots  &  \\
          &            &            &               R_N  \\
\end{array} \right),
\]
$\bold{D} \in \mathbb{R}^{(N+1)n \times (N+1)n}$ and $ \bold{R} \in \mathbb{R}^{p \times p}$.
The state update \eqref{eq:state_approx} may then be written as
\begin{equation*}\label{eq:update_4dform}
\bx^{(l+1)} = \bx^{(l)} + \delta \bx^{(l)},
\end{equation*}  
and the quadratic cost function \eqref{eq:incr_state} becomes 
\begin{equation}\label{eq:incr_wc4d-var_state}
J^{\delta} (\delta \bx^{(l)}) = \frac{1}{2} || \bL^{(l)} \bdx^{(l)} - \bold{b}^{(l)} ||^2_{\bold{D}^{-1}} + \frac{1}{2} || \bold{H}^{(l)} \bdx^{(l)} - \bold{d}^{(l)} ||^2_{\bold{R}^{-1}},
\end{equation}
where $||\bold{a}||^2_{\bold{A}^{-1}}=\bold{a}^T\bold{A}^{-1}\bold{a}$.
We omit the superscript $(l)$ for the outer iteration in the subsequent discussions. The minimum of \eqref{eq:incr_wc4d-var_state} can be found by solving linear systems. We discuss different formulations of these in the next three subsections.

\subsubsection{$\bold{3 \times 3}$ block saddle point formulation}\label{sec:3x3}
In pursuance of exploiting parallel computations in data assimilation, Fisher and G{\"u}rol \cite{Fisher17} proposed obtaining the state increment $\bdx$ by solving a saddle point system (see also Freitag and Green\cite{Freitag2018}). New variables are introduced
\begin{align}
\boldsymbol{\lambda} & = \bold{D}^{-1} (\bold{b} - \bL \bdx) \ \in \mathbb{R}^{(N+1)n}, \label{lambda_var}\\  
\boldsymbol{\mu} & = \bold{R}^{-1} (\bold{d} - \bold{H}\bdx) \ \in \mathbb{R}^p. \label{mu_var}
\end{align}
The gradient of the cost function \eqref{eq:incr_wc4d-var_state} with respect to $\bdx$ provides the optimality constraint 
\begin{align}
\bold{0} = & \bL^T \bold{D}^{-1} (\bL \bdx - \bold{b}) + \bold{H}^T \bold{R}^{-1} ( \bold{H} \bdx- \bold{d})  \nonumber \\
=& -(\bL^T \boldsymbol{\lambda} + \bold{H}^T \boldsymbol{\mu}). \label{opt_cond}
\end{align}
Multiplying \eqref{lambda_var} by $\bold{D}$ and \eqref{mu_var} by $\bold{R}$ together with \eqref{opt_cond}, yields a coupled linear system of equations:

\begin{equation}\label{eq:saddle}
 \mathcal{A}_3  \left( \begin{array}{c}
\boldsymbol{\lambda} \\
\boldsymbol{\mu} \\
\bdx \\
\end{array} \right)
= \left( \begin{array}{c}
\bold{b} \\
\bold{d} \\
\bold{0} \\
\end{array} \right), 
\end{equation}
where the coefficient matrix is given by
\begin{equation}\label{eq:saddle_mtrx}
\mathcal{A}_3  = \left( \begin{array}{ccc}
\bold{D} & \bold{0} & \bold{L} \\
\bold{0} & \bold{R} & \bold{H} \\
\bold{L}^T & \bold{H}^T & \bold{0} \\
\end{array} \right) \in \mathbb{R}^{(2(N+1)n + p) \times (2(N+1)n + p) }. 
\end{equation}

$\mathcal{A}_3$ is a sparse symmetric indefinite saddle point matrix that has a $3 \times 3$ block form. Such systems are explored in the optimization literature \cite{Greif2014,Morini2016,Morini2017}. When solving these systems iteratively, it is usually assumed that calculations involving the blocks on the diagonal are computationally expensive, while the off-diagonal blocks are cheap to apply and easily approximated. However, in our application, operations with the diagonal blocks are relatively cheap and the off-diagonal blocks are computationally expensive, particularly because of the integrations of the model and its adjoint in $\bold{L}$ and $\bold{L}^T$.  

Recall that the sizes of the blocks $\bold{R}$, $\bold{H}$ and $\bold{H}^T$ depend on the number of observations $p$. Thus, the size of $\mathcal{A}_3$ and possibly some of its characteristics are also affected by $p$. The saddle point systems that arise in different outer loops vary in the right hand sides and the linearisation states of $\bold{L}$ and $\bold{H}$. 

Because of the memory requirements of sparse direct solvers, they cannot be used to solve the $3 \times 3$ block saddle point systems that arise in an operational setting. Iterative solvers (such as MINRES, SYMMLQ \cite{Paige1975}, GMRES \cite{Saad1986}) need to be used. These Krylov subspace methods require matrix-vector products with $\mathcal{A}_3$ at each iteration. Note that the matrix-vector product $\mathcal{A}_3 \bold{q}$,  $\bold{q}^T=(q_1^T, q_2^T, q_3^T), \  q_1, q_3 \in \mathbb{R}^{(N+1)n}, q_2 \in \mathbb{R}^p$, involves multiplying $\bold{D}$ and $\bold{L}^T$ by $q_1$, $\bold{R}$ and $\bold{H}^T$ by $q_2$, and $\bold{L}$ and  $\bold{H}$ by $q_3$. These matrix-vector products may be performed in parallel. Furthermore, multiplication of each component of each block matrix with the respective part of the vector $q_i$ can be performed in parallel. The possibility of multiplying a vector with the blocks in $\bold{L}$ and $\bold{L}^T$ in parallel is particularly attractive, because the expensive operations involving the linearised model $M_i$ and its adjoint $M^T_i$ can be done at the same time for every $i \in \{0, 1, \dots, N-1 \}$.

\subsubsection{$\bold{2 \times 2}$ block saddle point formulation}\label{sec:2x2}
The saddle point systems with $3 \times3$ block coefficient matrices that arise from interior point methods are often reduced to $2 \times 2$ block systems \cite{Greif2014,Morini2017}. The $2 \times 2$ block formulation has not been used in data assimilation before. Because of its smaller size, it may be advantageous. Therefore, we now explore using this approach in the weak constraint 4D-Var setting.

Multiplying equation \eqref{lambda_var} by $\bold{D}$ and eliminating $\boldsymbol{\mu}$ in \eqref{opt_cond} gives the following equivalent system of equations
\begin{equation}\label{2block_system}
\mathcal{A}_2 \left( \begin{array}{c}
\boldsymbol{\lambda} \\
\bdx \\
\end{array} \right)
= \left( \begin{array}{c}
\bold{b} \\
-\bold{H}^T  \bold{R}^{-1} \bold{d} \\
\end{array} \right), 
\end{equation}
where
\begin{equation}\label{2block_mtrx}
\mathcal{A}_2 =  \left( \begin{array}{cc}
\bold{D}  & \bL \\
\bL^T & -\bold{H}^T  \bold{R}^{-1} \bold{H}
\end{array} \right) \in \mathbb{R}^{2(N+1)n \times 2(N+1)n }.
\end{equation}

The reduced matrix $\mathcal{A}_2 $ is a sparse symmetric indefinite saddle point matrix with a $2 \times 2$ block form. Unlike the $3 \times 3$ block matrix $\mathcal{A}_3$, its size is independent of the number of observations. 
$\mathcal{A}_2 $ involves the matrix $\bold{R}^{-1}$, which is usually available in data assimilation applications. The computationally most expensive blocks $\bL$ and $\bL^T$ are again the off-diagonal blocks.

Solving \eqref{2block_system} in parallel might be less appealing compared to solving \eqref{eq:saddle} in parallel: for a Krylov subspace method, the $(2,2)$ block $-\bold{H}^T  \bold{R}^{-1} \bold{H}$ need not be formed separately, i.e. only operators to perform the matrix-vector products with $\bold{H}^T$, $\bold{R}^{-1}$ and $\bold{H}$ need to be stored. Hence, a matrix-vector product $\mathcal{A}_2 \bold{q}$,  $\bold{q}^T=(q_1^T, q_3^T), \  q_1, q_3 \in \mathbb{R}^{(N+1)n}$, requires multiplying $\bold{D}$ and $\bold{L}^T$ by $q_1$, $\bold{L}$ and $\bold{H}$ by $q_3$ (which may be done in parallel) and subsequently $\bold{R}^{-1}$ by $\bold{H} q_3$, followed by $-\bold{H}^T$ by $\bold{R}^{-1} \bold{H} q_3$. Hence, the cost of matrix-vector products for the $3 \times 3$ and $2 \times 2$ block formulations differs in that the former needs matrix-vector products with $\bold{R}$ while the latter requires products with $\bold{R}^{-1}$, and the $2 \times 2$ block formulation requires some sequential calculations. However, notice that the expensive calculations that involve applying the operators $\bL$ and $\bL^T$ (the linearised model and its adjoint) can still be performed in parallel.

\subsubsection{$\bold{1 \times 1}$ block formulation}\label{sec:1x1}
The $2 \times 2$ block system can be further reduced to a $1 \times 1$ block system, that is, to the standard formulation (see, e.g., Tr{\'{e}}molet \cite{Tremolet06} and A. El-Said \cite{AES_thesis} for a more detailed consideration):
\begin{equation}\label{eq:1block_system}
(\bL^T \bold{D}^{-1} \bL + \bold{H}^T  \bold{R}^{-1} \bold{H}) \bdx = \bL^T \bold{D}^{-1} \bold{b}  + \bold{H}^T  \bold{R}^{-1} \bold{d}. 
\end{equation}
Observe that the coefficient matrix 
\begin{align}\label{1block_mtrx}
\mathcal{A}_1  = & \  \bL^T \bold{D}^{-1} \bL + \bold{H}^T  \bold{R}^{-1} \bold{H} \\
					  = &\  (\bL^T \quad \bold{H}^T ) 
					  		\left( \begin{array}{cc}
							\bold{D}^{-1}  & \bold{0} \\ 
 							\bold{0} &  \bold{R}^{-1}
						   	\end{array} \right)  
						   	\left( \begin{array}{c}
							\bL \\
							\bold{H}
							\end{array} \right) \nonumber
\end{align}
is the negative Schur complement of $\left( \begin{array}{cc}
\bold{D} & \bold{0} \\
\bold{0} & \bold{R}\\
\end{array} \right) $ in $\mathcal{A}_3$. The matrix $\mathcal{A}_1$ is block tridiagonal and symmetric positive definite, hence the conjugate gradient method by Hestenes and Stiefel \cite{Hestenes52} can be used. The computations with the linearised model in $\bold{L}$ at every time step can again be performed in parallel. However, the adjoint of the linearised model in $\bold{L}^T$ can only be applied after the computations with the model are finished, thus limiting the potential for parallelism. 

\section{Eigenvalues of the saddle point formulations}\label{sec:eigenvalues_theory}
One factor that influences the rate of convergence of Krylov subspace iterative solvers for symmetric systems is the spectrum of the coefficient matrix (see, for example, Section 9 in the survey paper\cite{Benzi2005} and Lectures 35 and 38 in the textbook\cite{TrefethenBau} for a discussion). Simoncini and Szyld \cite{SimonciniSzyld2013} have shown that any eigenvalues of the saddle point systems that lie close to zero can cause the iterative solver MINRES to stagnate for a number of iterations while the rate of convergence can improve if the eigenvalues are clustered. 

In the following, we examine how the eigenvalues of the block matrices $\mathcal{A}_3$, $\mathcal{A}_2$, and $\mathcal{A}_1$ change when new observations are added. This is done by considering the shift in the extreme eigenvalues of these matrices, that is the smallest and largest positive and negative eigenvalues. We then present bounds for the eigenvalues of these matrices. 

\subsection{Preliminaries}
In order to determine how changing the number of observations influences the spectra of $\mathcal{A}_3$, $\mathcal{A}_2$, and $\mathcal{A}_1$, we explore the extreme singular values and eigenvalues of some blocks in $\mathcal{A}_3$, $\mathcal{A}_2$ and $\mathcal{A}_1$. We state two theorems that we will require. Here we employ the notation $\lambda_k(A)$ to denote the $k$-th largest eigenvalue of a matrix $A$ and subscripts $min$ and $max$ are used to denote the smallest and largest eigenvalues, respectively.

\begin{theorem}[See Section 8.1.2 of Golub and Van Loan\cite{Golub13}]\label{th:eig_bounds}
If $A$ and $C$ are $n \times n$ Hermitian matrices, then 
\begin{center}
$\lambda_k(A) + \lambda_{min}(C) \leq \lambda_k(A+C) \leq \lambda_k(A) + \lambda_{max}(C),\quad k \in \{ 1,2,\dots, n\}$.
\end{center}
\end{theorem}

\begin{theorem}[Cauchy's Interlace Theorem, see Theorem 4.2 in Chapter 4 of Stewart and Sun\cite{Stewart1990}]\label{th:cauchys_interlace}
If $A$ is an $n \times n$ Hermitian matrix and $C$ is a $(n-1) \times (n-1)$ principal submatrix of $A$ (a matrix obtained by eliminating a row and a corresponding column of $A$), then 
\begin{center}
$\lambda_n(A) \leq \lambda_{n-1} (C) \leq \lambda_{n-1}(A) \leq \dots \leq \lambda_2(A) \leq \lambda_1(C) \leq \lambda_1(A)$. 
\end{center}
\end{theorem}

In the following lemmas we describe how the smallest and largest singular values of $(\bold{L}^T\ \bold{H}^T)$ (here $\bold{L}$ and $\bold{H}$ are as defined in Section \ref{sec:incremental}) and the extreme eigenvalues of the observation error covariance matrix $\bold{R}$ change when new observations are introduced. The same is done for the largest eigenvalues of $\bold{H}^T \bold{R}^{-1} \bold{H}$ assuming that $\bold{R}$ is diagonal. In these lemmas the subscript $k \in \{ 0, 1,\dots, (N+1)n-1\}$ denotes the number of available observations and the subscript $k+1$ indicates that a new observation is added to the system with $k$ observations, i.e. matrices $\bold{R}_k \in \mathbb{R}^{k \times k}$ and $\bold{H}_k \in \mathbb{R}^{k \times (N+1)n}$ correspond to a system with $k$ observations and $\bold{R}_{k+1}$ and $\bold{H}_{k+1}$ correspond to the system with an additional observation. We write $\bold{R}_{k+1} = \left( \begin{array}{cc}
\bold{R}_k & r \\
 r^T & \alpha \\
\end{array} \right)$ and $\bold{H}_{k+1} = \left(  \begin{matrix}
\bold{H}_k \\
h^T_{k+1}
\end{matrix}  \right)$, where $r \in \mathbb{R}^k$, $\alpha \in \mathbb{R}^1$, $\alpha > 0$ and $h_{k+1} \in \mathbb{R}^{(N+1)n}$ correspond to the new observation.

\begin{lemma}\label{comment:theta_change}
Let $\omega_{min}$ and $\omega_{max}$ be the smallest and largest singular values of $(\bold{L}^T\ \bold{H}^T_k)$, and $\phi_{min}$ and $\phi_{max}$ be the smallest and largest singular values of $(\bold{L}^T\ \bold{H}^T_{k+1})$. Then
\begin{gather*}\label{eq:LH_sing_val}
\omega_{min}^2 \leq \phi_{min}^2 \quad \text{and} \quad \omega_{max}^2 \leq \phi_{max}^2
\end{gather*}
i.e. the smallest and largest singular values of $(\bold{L}^T\ \bold{H}^T)$ increase or are unchanged when new observations are added. 
\end{lemma}

\begin{proof}
We consider the eigenvalues of $\bL^T \bL + \bold{H}^T_k \bold{H}_k$ and $\bL^T \bL + \bold{H}^T_{k+1} \bold{H}_{k+1}$, which are the squares of the singular values of $(\bold{L}^T\ \bold{H}^T_k)$ and $(\bold{L}^T\ \bold{H}^T_{k+1})$, respectively (see Section 2.4.2 of Golub and Van Loan \cite{Golub13}). We can write
\begin{equation*}
\bold{H}^T_{k+1} \bold{H}_{k+1}  = \left(  \begin{matrix}
\bold{H}_k ^T & h_{k+1} \end{matrix}  \right)
\left(  \begin{matrix}
\bold{H}_k \\
h^T_{k+1}
\end{matrix}  \right) = 
\bold{H}_k ^T \bold{H}_k +h_{k+1} h_{k+1}^T.
\end{equation*}
Then by Theorem \ref{th:eig_bounds}, 
\begin{equation*}
\omega_{min}^2 + \lambda_{min}( h_{k+1} h_{k+1}^T) \leq \phi_{min}^2, \quad k \in \{ 0,1,\dots, (N+1)n-1\},
\end{equation*}
and since $h_{k+1} h_{k+1}^T$ is a rank 1 symmetric positive semidefinite matrix, $\lambda_{min}( h_{k+1} h_{k+1}^T) = 0$. 

The proof for the largest singular values is analogous. 
\end{proof}
\begin{lemma}\label{comment:r_eigenvalues}
Consider the observation error covariance matrices $\bold{R}_k\in \mathbb{R}^{k \times k}$ and $\bold{R}_{k+1}\in \mathbb{R}^{(k +1)\times (k+1)}$. Then
\begin{gather*}
\lambda_{min}(\bold{R}_{k+1}) \leq \lambda_{min}(\bold{R}_k) \quad \text{and} \quad \lambda_{max}(\bold{R}_k) \leq \lambda_{max}(\bold{R}_{k+1}), \quad k \in \{0,1,\dots, (N+1)n-1\},
\end{gather*}
i.e. the largest (respectively, smallest) eigenvalue of $\bold{R}$ increases (respectively, decreases), or is unchanged when new observations are introduced.
\end{lemma}
\begin{proof}
When adding an observation, a row and a corresponding column are appended to $\bold{R}_k$ while the other entries of $\bold{R}_k$ are unchanged. The result is immediate by applying Theorem \ref{th:cauchys_interlace}.
\end{proof}
\begin{lemma}\label{comment:HtRinvH_eigenvalues}
If the observation errors are uncorrelated, i.e. $\bold{R}$ is diagonal, then
\begin{equation*}
\lambda_{max}(\bold{H}^T_{k} \bold{R}_k^{-1} \bold{H}_{k}) \leq \lambda_{max}(\bold{H}^T_{k+1} \bold{R}_{k+1}^{-1} \bold{H}_{k+1}), \quad k \in \{0,1,\dots, (N+1)n-1\},
\end{equation*}
i.e. for diagonal $\bold{R}$, the largest eigenvalue of $\bold{H}^T \bold{R}^{-1} \bold{H}$ increases or is unchanged when new observations are introduced.
\end{lemma}

\begin{proof}
The proof is similar to that of Lemma \ref{comment:theta_change}. For diagonal $\bold{R}_{k+1}$:
\begin{equation*}
\bold{R}^{-1}_{k+1} = \left(  \begin{matrix}
\bold{R}^{-1}_k &  \\
& \alpha^{-1}\\
\end{matrix}  \right), \ \alpha > 0. 
\end{equation*}
Then 
\begin{equation*}
\bold{H}^T_{k+1} \bold{R}^{-1}_{k+1} \bold{H}_{k+1}  = \left(  \begin{matrix}
\bold{H}_k ^T & h_{k+1} \end{matrix}  \right)
\left(  \begin{matrix}
\bold{R}^{-1}_k &  \\
& \alpha^{-1}\\
\end{matrix}  \right)
\left(  \begin{matrix}
\bold{H}_k \\
h^T_{k+1}
\end{matrix}  \right) = 
\bold{H}_k ^T \bold{R}_k^{-1} \bold{H}_k +\alpha^{-1} h_{k+1} h_{k+1}^T.
\end{equation*}
Hence, by Theorem \ref{th:eig_bounds}, 
\begin{equation*}
\lambda_{max}(\bold{H}^T_{k} \bold{R}^{-1}_k \bold{H}_{k}) + \alpha^{-1}  \lambda_{min}(h_{k+1} h_{k+1}^T) \leq \lambda_{max}(\bold{H}^T_{k+1} \bold{R}^{-1}_{k+1} \bold{H}_{k+1}), \quad k \in \{ 0,1,\dots, (N+1)n-1\},
\end{equation*}
and since $\lambda_{min}( h_{k+1} h_{k+1}^T) = 0$ the result is proved.
\end{proof}

\subsection*{Notation}\label{sec:notation}
In the following, we use the notation given in Table~\ref{tab:notation_eigv_sv} for the eigenvalues of $\mathcal{A}_3$, $\mathcal{A}_2$ and $\mathcal{A}_1$, and the eigenvalues and singular values of the blocks within them. We use subscripts $min$ and $max$ to denote the smallest and largest eigenvalues or singular values, respectively, and $\theta_{min}$ denote the smallest non-zero singular value of $(\bL^T \quad \bold{H}^T )$. In addition, $||\cdot||$ denotes the $L_2$ norm.

\begin{table}[h]
\begin{tabular}{l|c|c|c|c|c|c}
Matrix          &  $\mathcal{A}_3$ &  $  \mathcal{A}_2$   &$ \mathcal{A}_1$  &$  \bold{D}$    &  $  \bold{H}^T  \bold{R}^{-1} \bold{H}$   & $   \bold{R} $  \\
\hline
Eigenvalue  &      $  \gamma_i  $            &   $\zeta_i $       & $\chi_i$ &$  \psi_i   $       &  $  \nu_i         $                                          &  $  \rho_i $   \\
\end{tabular}
\hspace*{3cm}  
\begin{tabular}{l|c|c}
Matrix                & $(\bL^T \quad \bold{H}^T ) $               &  $ \bL$  \\ 
\hline
Singular value                     &    $\theta_i  $                       & $ \sigma_i $ \\
 \end{tabular} 
 \captionof{table}{Notation for the eigenvalues and singular values.}
  \label{tab:notation_eigv_sv}
 \end{table}

We also use
\begin{gather} 
\tau_{min} = min\{\psi_{min}, \rho_{min}\}, \label{eq:tau_min}  \\
\tau_{max} = max\{\psi_{max}, \rho_{max}\}. \label{eq:tau_max}
\end{gather}

\subsection{Bounds for the $\bold{3 \times 3}$ block formulation}\label{sec:3x3bounds}
To determine the numbers of positive and negative eigenvalues of $\mathcal{A}_3$ given in~\eqref{eq:saddle_mtrx}, we write $\mathcal{A}_3$ as a congruence transformation
\begin{equation*}
\mathcal{A}_3  = \left( \begin{array}{ccc}
\bold{D} & \bold{0} & \bold{L} \\
\bold{0} & \bold{R} & \bold{H} \\
\bold{L}^T & \bold{H}^T & \bold{0} \\
\end{array} \right) 
=
\left( \begin{array}{ccc}
\bold{D}  & \bold{0} & \bold{0} \\
 \bold{0} & \bold{R} & \bold{0} \\
\bL^T &  \bold{H}^T & \bold{I}
\end{array} \right)
\left( \begin{array}{ccc}
\bold{D}^{-1}  & \bold{0} & \bold{0} \\
 \bold{0} & \bold{R}^{-1} & \bold{0} \\
\bold{0} & \bold{0} & -\bL^T \bold{D}^{-1} \bL -\bold{H}^T  \bold{R}^{-1} \bold{H}
\end{array} \right)
\left( \begin{array}{ccc}
\bold{D} & \bold{0} & \bold{L} \\
\bold{0} & \bold{R} & \bold{H} \\
\bold{0} & \bold{0} & \bold{I}
\end{array} \right)
=
\hat{\bold{L}}\hat{\bold{B}}\hat{\bold{L}}^T,
\end{equation*}
where $\bold{I} \in \mathbb{R}^{(N+1)n \times (N+1)n}$ is the identity matrix. Thus, by Sylvester's law of inertia (see Section 8.1.5 of Golub and Van Loan\cite{Golub13}), $\mathcal{A}_3$ and $\hat{\bold{B}}$ have the same inertia, i.e. the same number of positive, negative, and zero eigenvalues. Since the blocks $\bold{D}^{-1}$, $\bold{R}^{-1}$ and  $\bL^T \bold{D}^{-1} \bL +\bold{H}^T  \bold{R}^{-1} \bold{H}=\mathcal{A}_1$ are symmetric positive definite matrices, $\mathcal{A}_3$ has $(N+1)n+p$ positive and $(N+1)n$ negative eigenvalues. In the following theorem, we explore how the extreme eigenvalues of $\mathcal{A}_3$ change when new observations are introduced.

\begin{theorem}\label{th:3x3eig_change}
The smallest and largest negative eigenvalues of $\mathcal{A}_3$ either move away from zero or are unchanged when new observations are introduced. 
The same holds for the largest positive eigenvalue, while the smallest positive eigenvalue approaches zero or is unchanged.
\end{theorem}

\begin{proof}
Let $\mathcal{A}_{3,k}$ denote $\mathcal{A}_3$ where $p=k$. To account for an additional observation, a row and a corresponding column is added to $\mathcal{A}_3$, hence $\mathcal{A}_{3,k}$ is a principal submatrix of $\mathcal{A}_{3,k+1}$. Let 
\begin{gather*}
\lambda_{-(N + 1)n}(\mathcal{A}_{3,k}) \leq \lambda_{-((N + 1)n-1)}(\mathcal{A}_{3,k}) \leq \cdots \leq \lambda_{-1}(\mathcal{A}_{3,k}) < 0 <  \lambda_{1}(\mathcal{A}_{3,k}) \leq \cdots \leq \lambda_{(N + 1)n+k}(\mathcal{A}_{3,k})
\end{gather*}
be the eigenvalues of $\mathcal{A}_{3,k}$, and 
\begin{gather*}
\lambda_{-(N + 1)n}(\mathcal{A}_{3,k+1}) \leq \lambda_{-((N + 1)n-1)}(\mathcal{A}_{3,k+1}) \leq \cdots \leq \lambda_{-1}(\mathcal{A}_{3,k+1}) < 0 <  \lambda_{1}(\mathcal{A}_{3,k+1}) \leq \cdots \leq \lambda_{(N + 1)n+k+1}(\mathcal{A}_{3,k+1})
\end{gather*}
be the eigenvalues of $\mathcal{A}_{3,k+1}$. Then by Theorem~\ref{th:cauchys_interlace}:
\begin{gather*}
\text{smallest negative eigenvalues}:  \quad  \lambda_{-(N + 1)n}(\mathcal{A}_{3,k+1}) \leq \lambda_{-(N + 1)n}(\mathcal{A}_{3,k}), \\
\text{largest negative eigenvalues}: \quad \lambda_{-1}(\mathcal{A}_{3,k+1})  \leq \lambda_{-1}(\mathcal{A}_{3,k}), \\
\text{smallest positive eigenvalues}: \quad \lambda_{1}(\mathcal{A}_{3,k+1})  \leq \lambda_{1}(\mathcal{A}_{3,k}),\\
\text{largest positive eigenvalues}: \quad \lambda_{(N + 1)n+k}(\mathcal{A}_{3,k})  \leq \lambda_{(N + 1)n+k+1}(\mathcal{A}_{3,k+1}).
\end{gather*}
\end{proof}

To obtain information on not only how the eigenvalues of $\mathcal{A}_3$ change because of new observations, but also on where the eigenvalues lie when the number of observations is fixed, we formulate intervals for the negative and positive eigenvalues of $\mathcal{A}_3$.

\begin{theorem}\label{the:3x3eig}
The negative eigenvalues of $\mathcal{A}_3$ lie in the interval
\begin{equation}\label{3blockNeg}
I_-  =  \Big[ \frac{1}{2} \Big(\tau_{min} - \sqrt{\tau_{min}^2 +4 \theta_{max}^2}\Big), \frac{1}{2} \Big(\tau_{max} - \sqrt{\tau_{max}^2 +4 \theta_{min}^2} \Big) \Big] 
\end{equation}
and the positive eigenvalues lie in the interval
\begin{equation}\label{3blockPos}
 I_+ = \Big[ \tau_{min}, \frac{1}{2} \Big(\tau_{max} + \sqrt{\tau_{max}^2 +4 \theta_{max}^2}\Big) \Big],
\end{equation}
where 
$\tau_{min}, \tau_{max}$, and $\theta_i$ are defined in \eqref{eq:tau_min}, \eqref{eq:tau_max}, and Table~\ref{tab:notation_eigv_sv}.
\end{theorem}

\begin{proof}
Lemma 2.1 of Rusten and Winther\cite{RustenWinther1992} gives eigenvalue intervals for matrices of the form $A = \left( \begin{array}{cc}
C & E \\
E^T & 0\\
\end{array} \right)$. 
Applying these intervals in the case 
$C = \left( \begin{array}{cc}
\bold{D} & \bold{0} \\
\bold{0} & \bold{R}\\
\end{array} \right) $ and 
$E^T = \left( \begin{array}{cc}
\bold{L}^T & \bold{H}^T
\end{array} \right)$ 
yields the required results. 
\end{proof}

We present two corollaries that describe how the bounds in Theorem \ref{the:3x3eig} change if additional observations are introduced and conclude that the change of the bounds is consistent with the results in Theorem~\ref{th:3x3eig_change}.
\begin{cor}\label{cor:3block_pos_bounds}
The interval for the positive eigenvalues of $\mathcal{A}_3$ in \eqref{3blockPos} either expands or is unchanged when new observations are added.
\end{cor}
\begin{proof}
First, consider the positive upper bound $\frac{1}{2} \Big(\tau_{max} + \sqrt{\tau_{max}^2 +4 \theta_{max}^2}\Big)$. By Lemma \ref{comment:theta_change}, $\theta_{max}^2$ increases or is unchanged when additional observations are included. If $\tau_{max} = \rho_{max}$, the same holds for $\tau_{max}$ (by Lemma \ref{comment:r_eigenvalues}). If $\tau_{max} = \psi_{max}$, changing the number of observations does not affect $\tau_{max}$. Hence, the positive upper bound increases or is unchanged.

The positive lower bound $\tau_{min}$ is unaltered if $\tau_{min} = \psi_{min}$. If $\tau_{min} = \rho_{min}$, the bound decreases or is unchanged by Lemma \ref{comment:r_eigenvalues}.
\end{proof}
\begin{cor}\label{cor:3block_neg_bounds}
If $\tau_{max}=\psi_{max}$, the upper bound for the negative eigenvalues of $\mathcal{A}_3$ in \eqref{3blockNeg} is either unchanged or moves away from zero when new observations are added. If $\tau_{min}=\psi_{min}$, the same holds for the lower bound for negative eigenvalues in \eqref{3blockNeg}.
\end{cor}
\begin{proof}
The results follow from the facts that $\psi_{max}$ and $\psi_{min}$ do not change if observations are added, whereas $\theta_{min}$ and $\theta_{max}$ increase or are unchanged by Lemma \ref{comment:theta_change}.
\end{proof}

If $\tau_{max}=\rho_{max}$ or $\tau_{min}=\rho_{min}$, it is unclear how the interval for the negative eigenvalues in \eqref{3blockNeg} changes, because
$\sqrt{\tau_{min}^2 +4 \theta_{max}^2}$ can increase, decrease or be unchanged, and both $\tau_{max}$ and $\sqrt{\tau_{max}^2 +4 \theta_{min}^2}$ can increase or be unchanged.

\subsection{Bounds for the $\bold{2 \times 2}$ block formulation}\label{sec:2x2bounds}
$\mathcal{A}_2$ given in~\eqref{2block_mtrx} is equal to the following congruence transformation
\begin{equation*}
\mathcal{A}_2  =  \left( \begin{array}{cc}
\bold{D}  & \bL \\
\bL^T & -\bold{H}^T  \bold{R}^{-1} \bold{H}
\end{array} \right) =
\left( \begin{array}{cc}
\bold{D}  & \bold{0} \\
\bL^T &  \bold{I}
\end{array} \right)
\left( \begin{array}{cc}
\bold{D}^{-1}  & \bold{0} \\
\bold{0} & -\bL^T \bold{D}^{-1} \bL -\bold{H}^T  \bold{R}^{-1} \bold{H}
\end{array} \right)
\left( \begin{array}{cc}
\bold{D}  & \bL \\
\bold{0} & \bold{I}
\end{array} \right), 
\end{equation*}
where $\bold{I} \in \mathbb{R}^{(N+1)n \times (N+1)n}$ is the identity matrix. Then by Sylvester's law, $\mathcal{A}_2$ has $(N+1)n$ positive and $(N+1)n$ negative eigenvalues. The change of the extreme negative and positive eigenvalues of $\mathcal{A}_2$ due to the additional observations is analysed in the subsequent theorem. However, the result holds only in the case of uncorrelated observation errors, unlike the general analysis for $\mathcal{A}_3$ in Theorem~\ref{th:3x3eig_change}.

\begin{theorem}\label{th:eigenvalues_2x2}
If the observation errors are uncorrelated, i.e. $\bold{R}$ is diagonal, then the smallest and largest negative eigenvalues of $\mathcal{A}_2$ either move away from zero or are unchanged when new observations are added. Contrarily, the smallest and largest positive eigenvalues of $\mathcal{A}_2$ approach zero or are unchanged.

\end{theorem}
\begin{proof}
Matrices $\bold{D}$ and $\bL$ do not depend on the number of observations. In Lemma~\ref{comment:HtRinvH_eigenvalues}, we have shown that $\bold{H}^T_{k+1} \bold{R}^{-1}_{k+1} \bold{H}_{k+1}   = \bold{H}_k ^T \bold{R}_k^{-1} \bold{H}_k +\alpha^{-1} h_{k+1} h_{k+1}^T, \ (\alpha > 0)$ for diagonal $\bold{R}$. Hence, when $\mathcal{A}_{2,k}$ denotes $\mathcal{A}_2$ with $p=k$, we can write 
\begin{equation*}
\mathcal{A}_{2,k+1} = \mathcal{A}_{2,k} + \left( \begin{array}{cc}
\bold{0} & \bold{0} \\
\bold{0} & -\alpha^{-1} h_{k+1} h_{k+1}^T 
\end{array} \right) = \mathcal{A}_{2,k} + \mathcal{E}_2,
\end{equation*}
where $\mathcal{E}_2$ has negative and zero eigenvalues. 
Let 
\begin{gather*}
\lambda_{-(N + 1)n}(\mathcal{A}_{2,k}) \leq  \cdots \leq \lambda_{-1}(\mathcal{A}_{2,k}) < 0 <  \lambda_{1}(\mathcal{A}_{2,k}) \leq \cdots \leq \lambda_{(N + 1)n}(\mathcal{A}_{2,k})
\end{gather*}
be the eigenvalues of $\mathcal{A}_{2,k}$, and 
\begin{gather*}
\lambda_{-(N + 1)n}(\mathcal{A}_{2,k+1}) \leq \cdots \leq \lambda_{-1}(\mathcal{A}_{2,k+1}) < 0 <  \lambda_{1}(\mathcal{A}_{2,k+1}) \leq \cdots \leq \lambda_{(N + 1)n}(\mathcal{A}_{2,k+1})
\end{gather*}
be the eigenvalues of $\mathcal{A}_{2,k+1}$. By Theorem~\ref{th:eig_bounds}, 

\begin{gather*}
\text{smallest negative eigenvalues}:  \quad \lambda_{-(N + 1)n}(\mathcal{A}_{2,k}) - \alpha^{-1} \lambda_{max} ( h_{k+1} h_{k+1}^T) \leq \lambda_{-(N + 1)n}(\mathcal{A}_{2,k+1}) \leq \lambda_{-(N + 1)n}(\mathcal{A}_{2,k}), \\
\text{largest negative eigenvalues}: \quad \lambda_{-1}(\mathcal{A}_{2,k}) - \alpha^{-1} \lambda_{max} ( h_{k+1} h_{k+1}^T) \leq \lambda_{-1}(\mathcal{A}_{2,k+1}) \leq \lambda_{-1}(\mathcal{A}_{2,k}), \\
\text{smallest positive eigenvalues}: \quad  \lambda_{1}(\mathcal{A}_{2,k}) - \alpha^{-1} \lambda_{max} ( h_{k+1} h_{k+1}^T) \leq \lambda_{1}(\mathcal{A}_{2,k+1}) \leq \lambda_{1}(\mathcal{A}_{2,k}),\\
\text{largest positive eigenvalues}:  \quad \lambda_{(N + 1)n}(\mathcal{A}_{2,k}) - \alpha^{-1} \lambda_{max} ( h_{k+1} h_{k+1}^T) \leq \lambda_{(N + 1)n}(\mathcal{A}_{2,k+1}) \leq \lambda_{(N + 1)n}(\mathcal{A}_{2,k}).
\end{gather*}
\end{proof}

We further search for the intervals in which the negative and positive eigenvalues of $\mathcal{A}_2$ lie. We follow a similar line of thought as in Silvester and Wathen \cite{SilvesterWathen1994}, with the energy arguments for any non-zero vector $\bold{w} \in \mathbb{R}^{(N+1)n}$
\begin{gather}
\psi_{min} ||\bold{w}||^2  \leq \bold{w}^T \bold{D} \bold{w}  \leq  \psi_{max} ||\bold{w}||^2, \label{energyD} \\
-\nu_{max} ||\bold{w}||^2  \leq  - \bold{w}^T  \bold{H}^T  \bold{R}^{-1} \bold{H} \bold{w}  \leq  -\nu_{min} ||\bold{w}||^2, \label{energy-HTRiH}\\
\sigma_{min} ||\bold{w}||  \leq  ||\bold{L}^T \bold{w}||  \leq \sigma_{max} ||\bold{w}||, \label{energyL} \\
\theta_{min} ||\bold{w}||  \leq  ||(\bold{L}^T \  \bold{H}^T)^T \bold{w}||  \leq \theta_{max} ||\bold{w}||. \label{energyLH}
\end{gather}

\begin{theorem}\label{2blockintervals}
The negative eigenvalues of $\mathcal{A}_2$ lie in the interval
\begin{equation}\label{th:2blockNegIntervals}
I_- =  \left[ \frac{1}{2} \left( \psi_ {min}- \nu_{max} - \sqrt{(\psi_{min} + \nu_{max})^2 + 4\sigma_{max}^2}\right),  min \left\{\beta_1 , max \left\{ \beta_2, \beta_3 \right \} \right \} \right], 
 \end{equation}
 where
 \begin{align}
 \beta_1 & =  \frac{1}{2} \left(\psi_{max} - \nu_{min} - \sqrt{(\psi_{max}+\nu_{min})^2+4\sigma_{min}^2} \right), \label{2bnegupper1}\\
 \beta_2 & = -\rho_{max}^{-1}\theta^2_{min}, \label{2bnegupper2} \\
 \beta_3 & = \frac{1}{2} \left(\psi_{max} - \sqrt{\psi_{max}^2+4\theta_{min}^2} \right), \label{2bnegupper3}
 \end{align}
and the positive ones lie in the interval
\begin{equation}\label{th:2blockPosIntervals}
I_+ = \left[  \frac{1}{2} \left( \psi_{min} - \nu_{max} + \sqrt{(\psi_{min} + \nu_{max})^2 + 4 \sigma^2_{min}} \right), \frac{1}{2} \left( \psi_{max} - \nu_{min} + \sqrt{(\psi_{max} + \nu_{min} )^2 +4 \sigma^2_{max}}\right) \right].
\end{equation}
\end{theorem}

\begin{proof}
Assume that $(\bold{u}^T, \bold{v}^T)^T, \ \bold{u}, \bold{v} \in \mathbb{R}^{(N+1)n}$ is an eigenvector of $\mathcal{A}_2$ with an eigenvalue $\zeta$. Then the eigenvalue equations are
\begin{align}
\bold{D} \bold{u} + \bL \bold{v} & = \zeta \bold{u}, \label{2blockEigEq1} \\
\bold{L}^T \bold{u} -  \bold{H}^T  \bold{R}^{-1} \bold{H} \bold{v} &= \zeta \bold{v}. \label{2blockEigEq2}
\end{align}
We note that if $\bold{u} = \bold{0}$ then $\bold{v} = \bold{0}$ by \eqref{2blockEigEq1} and if $\bold{v} = \bold{0}$ then $\bold{u} = \bold{0}$ by \eqref{2blockEigEq2}. Hence, $\bold{u}, \bold{v}  \neq \bold{0}$.

First, we consider $\zeta > 0$. Equation \eqref{2blockEigEq2} gives $\bold{v} = (\bold{I} \zeta + \bold{H}^T  \bold{R}^{-1} \bold{H})^{-1}\bold{L}^T \bold{u}$, where $\bold{I} \in \mathbb{R}^{(N+1)n \times (N+1)n}$. The matrix \mbox{$\bold{I} \zeta + \bold{H}^T  \bold{R}^{-1} \bold{H}$} is positive definite, hence nonsingular. We multiply \eqref{2blockEigEq1} by $\bold{u}^T$ and use the previous expression for $\bold{v}$ to get
\begin{equation}\label{uTsumEigEq}
\bold{u}^T \bold{D} \bold{u} + \bold{u}^T \bold{L}  (\bold{I} \zeta + \bold{H}^T  \bold{R}^{-1} \bold{H})^{-1} \bold{L}^T \bold{u} = \zeta ||\bold{u}||^2.
\end{equation}
The eigenvalues of $(\bold{I} \zeta + \bold{H}^T  \bold{R}^{-1} \bold{H})^{-1}$ in increasing order are $(\zeta + \nu_{max})^{-1}, \dots, (\zeta + \nu_{min})^{-1}$. Then
\begin{align*}
\bold{u}^T \bold{L}  (\bold{I} \zeta + \bold{H}^T  \bold{R}^{-1} \bold{H})^{-1} \bold{L}^T \bold{u} \geq & \frac{1}{\zeta + \nu_{max}} || \bold{L}^T \bold{u}||^2 \\
\geq & \frac{1}{\zeta + \nu_{max}}  \sigma_{min}^2 ||\bold{u}||^2 \quad \text{[by \eqref{energyL}]}.
\end{align*}
Hence, this inequality together with \eqref{energyD} and \eqref{uTsumEigEq} gives 
\begin{equation*}
\zeta ||\bold{u}||^2 \geq \psi_{min} ||\bold{u}||^2 + \frac{1}{\zeta + \nu_{max}}  \sigma_{min}^2 ||\bold{u}||^2
\end{equation*}
and solving
\begin{equation*}
\zeta^2 + (\nu_{max} - \psi_{min}) \zeta - \psi_{min} \nu_{max} - \sigma^2_{min} \geq 0
\end{equation*}
results in 
\begin{equation*}
\zeta \geq \frac{1}{2} \left( \psi_{min} - \nu_{max} + \sqrt{(\psi_{min} + \nu_{max})^2 + 4 \sigma^2_{min}} \right).
\end{equation*}

Similarly, using the upper bound from \eqref{energyD} and employing \eqref{uTsumEigEq} yields the upper bound
\begin{equation*}
\zeta \leq \frac{1}{2} \left( \psi_{max} - \nu_{min} + \sqrt{(\psi_{max} + \nu_{min})^2 + 4 \sigma^2_{max}} \right).
\end{equation*}

Now consider the case $\zeta < 0$. Since $\bold{D} - \zeta \bold{I}$ is positive definite, from \eqref{2blockEigEq1} \mbox{$\bold{u} = -(\bold{D} - \zeta \bold{I})^{-1} \bL \bold{v}$}. Using this expression and multiplying \eqref{2blockEigEq2} by $\bold{v}^T$ gives
\begin{equation}\label{2bv}
- \zeta ||\bold{v}||^2 = \bold{v}^T \bL^T (\bold{D} - \zeta \bold{I}_{(N+1)n})^{-1} \bL \bold{v} + \bold{v}^T \bold{H}^T  \bold{R}^{-1} \bold{H} \bold{v}.
\end{equation}
Then using \eqref{energy-HTRiH}, \eqref{energyL} and the fact that the smallest eigenvalue of $(\bold{D} - \zeta \bold{I})^{-1}$ is $(\psi_{max}-\zeta)^{-1}$ results in inequality
\begin{equation*}
- \zeta ||\bold{v}||^2 \geq  \sigma_{min}^2 ||\bold{v}||^2 \frac{1}{\psi_{max}-\zeta} +\nu_{min}||\bold{v}||^2,
\end{equation*}
which can be expressed as
\begin{equation*}
\zeta^2 - (\psi_{max} - \nu_{min})\zeta - \nu_{min}\psi_{max}-\sigma_{min}^2 \geq 0,
\end{equation*}
and its solution gives the upper bound 
\begin{equation}\label{2bnu1}
\zeta \leq \frac{1}{2} \left( \psi_{max} - \nu_{min} - \sqrt{(\psi_{max} + \nu_{min})^2+4\sigma_{min}^2} \right)=\beta_1.
\end{equation}

Notice that the bound \eqref{2bnu1} takes into account information on observations only if the system is fully observed. Otherwise, $p<(N+1)n$ and $\nu_{min}=0$. 

We obtain an alternative upper bound for the negative eigenvalues, that depends on the observational information and might be useful for the fully observed case, too. Equation \eqref{2bv} may be written as 
\begin{equation*}
- \zeta ||\bold{v}||^2 = \bold{v}^T (\bL^T \ \bold{H}^T ) \left( \begin{array}{cc}
(\bold{D} - \zeta \bold{I})^{-1} & \bold{0} \\ 
 \bold{0} &  \bold{R}^{-1}
\end{array} \right)  \left( \begin{array}{c}
\bL \\
\bold{H}\end{array} \right) \bold{v}.
\end{equation*}
Eigenvalues of the $2 \times 2$ block matrix in the previous equation are the eigenvalues of $(\bold{D} - \zeta \bold{I})^{-1}$ and $\bold{R}^{-1}$. Thus, by an energy argument \eqref{energyD},
\begin{align*}
- \zeta ||\bold{v}||^2 & \geq min\{\rho^{-1}_{max}, (-\zeta + \psi_{max})^{-1}\} || (\bL^T \ \bold{H}^T )^T \bold{v} ||^2 \\
 & \geq min\{\rho^{-1}_{max}, (-\zeta + \psi_{max})^{-1} \} \theta^2_{min} ||\bold{v}||^2 \quad \text{[by \eqref{energyLH}]} .
\end{align*} Hence,
\begin{equation*}
 \zeta \leq - \theta^2_{min} \iota, 
\end{equation*}
where $\iota = min\{\rho^{-1}_{max}, (-\zeta + \psi_{max})^{-1}\}$. If $\iota = \rho^{-1}_{max}$, the upper bound is 
\begin{equation*}
\zeta \leq  -\rho^{-1}_{max} \theta^2_{min}=\beta_2.
\end{equation*}
If $\iota =(-\zeta + \psi_{max})^{-1}$, the following inequality
\begin{equation*}
\zeta^2 - \psi_{max}\zeta - \theta^2_{min} \geq 0
\end{equation*}
gives the bound
\begin{equation*}
\zeta \leq \frac{1}{2} \left(\psi_{max} - \sqrt{\psi_{max}^2+4\theta_{min}^2} \right)=\beta_3.
\end{equation*}
Hence, 
\begin{equation}\label{2bnu2}
\zeta \leq max\{ \beta_2, \beta_3 \}.
\end{equation}
The required upper bound follows from \eqref{2bnu1} and \eqref{2bnu2}

Next, we obtain the lower bound for the negative eigenvalues. Using equation \eqref{2bv} with the largest eigenvalue of $ (\bold{D} - \zeta \bold{I})^{-1}$ and other parts of \eqref{energy-HTRiH} and \eqref{energyL} yields
\begin{equation*}
- \zeta ||\bold{v}||^2  \leq \sigma_{max}^2 ||\bold{v}||^2 \frac{1}{\psi_{min}-\zeta} +\nu_{max}||\bold{v}||^2.
\end{equation*} 
Solving 
\begin{equation*}
 \zeta^2 - (\psi_{min} - \nu_{max})\zeta - \nu_{max}\psi_{min}-\sigma_{max}^2 \leq 0
\end{equation*}
results in 
\begin{equation*}
\zeta \geq \frac{1}{2} \Big( \psi_{min} - \nu_{max} - \sqrt{(\psi_{min} +  \nu_{max})^2 + 4\sigma_{max}^2} \Big).
\end{equation*}
\end{proof}

We observe that if the system is not fully observed, then $p<(N+1)n$ and $\nu_{min}=0$, and the upper bound for the positive eigenvalues and the upper bound for the negative eigenvalues \eqref{2bnegupper1} in Theorem \ref{2blockintervals} reduces to (2.11) and (2.13) of Silvester and Wathen \cite{SilvesterWathen1994}. 

We are interested in how the bounds in Theorem \ref{2blockintervals} change if additional observations are introduced. The change to the upper negative bound in \eqref{th:2blockNegIntervals} depends on which of \eqref{2bnegupper1}, \eqref{2bnegupper2} or \eqref{2bnegupper3} gives the bound. Hence, in Corollary~\ref{cor:beta3morebeta2} we comment on when \eqref{2bnegupper3} is larger than \eqref{2bnegupper2} and Corollary~\ref{comment:upper_neg2x2} describes a setting when the negative upper bound is given by \eqref{2bnegupper3}.
\begin{cor}\label{cor:beta3morebeta2}
\begin{equation*}
max\{\beta_2, \beta_3\}=\beta_3 \quad \iff \frac{1}{2}(\psi_{max}+\sqrt{\psi_{max}^2 + \theta_{min}^2} ) \geq \rho_{max}.
\end{equation*}
\end{cor}

\begin{proof}
$max\{\beta_2, \beta_3\}=\beta_3$ if and only if
\begin{equation*}
\frac{1}{2} \Big(\psi_{max} - \sqrt{\psi_{max}^2+4\theta_{min}^2} \Big) \geq -\rho_{max}^{-1}\theta^2_{min}.
\end{equation*}
Rearranging this inequality gives
\begin{equation*}
\psi_{max} + 2\rho_{max}^{-1}\theta^2_{min} \geq \sqrt{\psi_{max}^2+4\theta_{min}^2}.
\end{equation*}
Squaring both sides with further rearrangement results in
\begin{equation*}
\theta_{min}^2 (\rho_{max}^{-1}\psi_{max}  + \rho_{max}^{-2}\theta^2_{min} -1)  \geq 0.
\end{equation*}
Since $\theta_{min}^2>0$, this is equivalent to
\begin{equation*}
\rho_{max}^{2} - \rho_{max}\psi_{max} - \theta^2_{min} \leq 0,
\end{equation*} 
from which it follows that
\begin{equation*}
\rho_{max} \leq \frac{1}{2} \Big(\psi_{max} + \sqrt{\psi_{max}^2+4\theta_{min}^2} \Big).
\end{equation*}
\end{proof}
Corollary~\ref{cor:beta3morebeta2} can be used to check if the assumption in the following corollary holds.
\begin{cor}\label{comment:upper_neg2x2}
If the system is not fully observed and $max\{\beta_2, \beta_3\}=\beta_3$, then the upper bound for the negative eigenvalues of $\mathcal{A}_2$ is given by \eqref{2bnegupper3}. 
\end{cor}

\begin{proof}
The singular values of $\bold{L}$ and $(\bL^T \quad \bold{H}^T )$ are the square roots of the eigenvalues of $\bL^T \bL$ and \mbox{$\bL^T \bL + \bold{H}^T \bold{H}$}, respectively. Hence, by Theorem \ref{th:eig_bounds}, 
\begin{equation*}
\sigma^2_{min} + \lambda_{min}(\bold{H}^T \bold{H}) \leq \theta^2_{min},
\end{equation*}
where $\lambda_{min}(\bold{H}^T \bold{H}) \geq 0$, since $\bold{H}^T \bold{H}$ is symmetric positive semidefinite.
Also, if $p< (N+1)n$, then $\bold{H}^T \bold{R}^{-1} \bold{H}$ is singular, i.e. $\nu_{min}=0$, and from \eqref{2bnegupper1} and \eqref{2bnegupper3}
\begin{equation*}
\beta_1  =  \frac{1}{2} \Big(\psi_{max}  - \sqrt{\psi_{max}^2+4\sigma_{min}^2} \Big) \geq \frac{1}{2} \Big(\psi_{max} - \sqrt{\psi_{max}^2+4\theta_{min}^2} \Big) = \beta_3 = max\{\beta_2,\beta_3\}.
\end{equation*}
\end{proof} 
We further describe how the negative upper bound changes if it is given by \eqref{2bnegupper1} or \eqref{2bnegupper3}, including the case described in Corollary~\ref{comment:upper_neg2x2}.
\begin{cor}\label{cor:2block_neg_upper}
If the upper bound for the negative eigenvalues of $\mathcal{A}_2$ in \eqref{th:2blockNegIntervals} is given by $\beta_1$  or $\beta_3$, then the bound moves away from zero or stays the same when new observations are added.
\end{cor}
\begin{proof}
$\beta_1$ does not change while the system is not fully observed. When the system becomes fully observed, $\nu_{min}>0$ and $\beta_1$ decreases.
$\beta_3$ decreases or stays the same by Lemma \ref{comment:theta_change}.
\end{proof}

Note that if the negative upper bound in \eqref{th:2blockNegIntervals} is given by $\beta_2$, it is unclear how the bound changes with the number of observations, since both $\rho_{max}$ and $\theta^2_{min}$ increase or stay the same. The same is true for the positive bounds in \eqref{th:2blockPosIntervals}. Only $\nu_{max}$ and $\nu_{min}$ depend on the available observations and they are contained in elements with positive and negative signs.

The result in Corollary~\ref{cor:2block_neg_upper} that applies for $\mathcal{A}_2$ with a general $\bold{R}$ is consistent with the result in Theorem~\ref{th:eigenvalues_2x2} that considers $\mathcal{A}_2$ with a diagonal $\bold{R}$. The same holds for the result in the following corollary, that determines how the lower bound for the negative eigenvalues of $\mathcal{A}_2$ changes in the special case of uncorrelated observational errors. 
\begin{cor}\label{cor:2block_neg_lower_diag_R}
If the observation error covariance matrix $\bold{R}$ is diagonal, the negative lower bound in \eqref{th:2blockNegIntervals} moves away from zero or stays the same when additional observations are introduced.
\end{cor}
\begin{proof}
The result follows by applying Lemma \ref{comment:HtRinvH_eigenvalues} to see how $\nu_{max}$ changes.
\end{proof}
In the following corollary, we consider the intervals for the positive eigenvalues of $\mathcal{A}_3$ and $\mathcal{A}_2$ with a fixed number of observations. It suggests that we may expect the positive eigenvalues of $\mathcal{A}_2$ to be more clustered than those of $\mathcal{A}_3$.
\begin{cor}\label{corollary:pos_eigen_bounds}
The interval for the positive eigenvalues of $\mathcal{A}_2$ is contained in the interval for the positive eigenvalues of $\mathcal{A}_3$, i.e. 
\begin{equation*}
\begin{split}
\Big[ \frac{1}{2} \Big(\psi_{min} -\nu_{max} + \sqrt{(\psi_{min}+\nu_{max})^2 +4 \sigma_{min}^2}\Big), \frac{1}{2} \Big( \psi_{max} - \nu_{min} + \sqrt{(\psi_{max}+\nu_{min})^2 +4 \sigma^2_{max}}\Big) \Big] \subseteq \\
 \Big[ \tau_{min}, \frac{1}{2} \Big(\tau_{max} + \sqrt{\tau_{max}^2 +4 \theta_{max}^2}\Big) \Big].
\end{split}
\end{equation*}
\end{cor}

\begin{proof}
As observed in Corollary \ref{comment:upper_neg2x2},
\begin{equation*}
\sigma^2_{max} + \lambda_{min}(\bold{H}^T \bold{H}) \leq \theta^2_{max},
\end{equation*}
with $\lambda_{min}(\bold{H}^T \bold{H}) \geq 0$. 
Also, by definition $\tau_{max} \geq \psi_{max}$ and the following inequality for the upper bound for the positive eigenvalues of $\mathcal{A}_3$ holds
\begin{equation*}
\frac{1}{2} \Big(\tau_{max} + \sqrt{\tau_{max}^2 +4 \theta_{max}^2}\Big) \geq \frac{1}{2} \Big(\psi_{max} + \sqrt{\psi_{max}^2 +4 \theta_{max}^2}\Big).
\end{equation*}
Thus, we show that the upper bound for positive eigenvalues of $\mathcal{A}_3$ is larger than the upper bound for positive eigenvalues of $\mathcal{A}_2$:
\begin{align}
\frac{1}{2} \Big(\psi_{max} + \sqrt{\psi_{max}^2 +4 \theta_{max}^2}\Big) & \geq \frac{1}{2} \Big( \psi_{max} - \nu_{min} + \sqrt{(\psi_{max}+\nu_{min})^2 +4 \sigma^2_{max}}\Big) \nonumber \\
\iff \nu_{min} + \sqrt{\psi_{max}^2 +4 \theta_{max}^2} & \geq \sqrt{(\psi_{max}+\nu_{min})^2 +4 \sigma^2_{max}} \nonumber \\
\text{(squaring both sides and simplifying)}\quad \iff 2\theta_{max}^2 +\nu_{min} \sqrt{\psi_{max}^2 +4 \theta_{max}^2} & \geq \psi_{max}\nu_{min} + 2 \sigma^2_{max} \nonumber \\
\text{(rearranging)}\quad \iff 2(\theta_{max}^2 - \sigma^2_{max}) & \geq \nu_{min}(\psi_{max} - \sqrt{\psi_{max}^2 +4 \theta_{max}^2}). \label{ineq:posUpperCompare}
\end{align}
Inequality \eqref{ineq:posUpperCompare} always holds because the left hand side is positive and the right hand side is negative.

We also show that the lower bound for the positive eigenvalues of $\mathcal{A}_3$ is smaller than the lower bound for the positive eigenvalues of $\mathcal{A}_2$:
\begin{equation*}
\tau_{min} \leq \frac{1}{2} \Big(\psi_{min} -\nu_{max} +\sqrt{(\psi_{min}+\nu_{max})^2 +4 \sigma_{min}^2}\Big).
\end{equation*}
Note that by definition $\tau_{min} \leq \psi_{min}$ and the following inequality always holds 
\begin{equation*}
\psi_{min} \leq \frac{1}{2} \Big(\psi_{min} -\nu_{max} + \sqrt{(\psi_{min}+\nu_{max})^2 +4 \sigma_{min}^2}\Big),
\end{equation*}
because it can be simplified to
\begin{align*}
\psi_{min} + \nu_{max} & \leq \sqrt{(\psi_{min}+\nu_{max})^2 +4 \sigma_{min}^2} \\
\text{(squaring both sides)}\quad \iff (\psi_{min} + \nu_{max})^2 & \leq (\psi_{min}+\nu_{max})^2 +4 \sigma_{min}^2 \\
\iff 0 & \leq 4 \sigma_{min}^2.
\end{align*}
\end{proof}

\subsection{Bounds for the $\bold{1 \times 1}$ block formulation}\label{sec:1x1bounds}
The system matrix $\mathcal{A}_1$ given by \eqref{1block_mtrx} is symmetric positive definite and so its eigenvalues are positive. We determine how these change due to additional observations when the observation errors are uncorrelated (as for the extreme eigenvalues of $\mathcal{A}_2$ in Theorem~\ref{th:eigenvalues_2x2}).

\begin{theorem}\label{th:eig1x1change}
If the observation errors are uncorrelated, i.e. $\bold{R}$ is diagonal, then the eigenvalues of $\mathcal{A}_1$ move away from zero or are unchanged when new observations are added.
\end{theorem}
\begin{proof}
Let $\mathcal{A}_{1,k}$ denote $\mathcal{A}_1$ where $p=k$. Then
$\mathcal{A}_{1,k+1} = \bL^T \bold{D}^{-1} \bL + \bold{H}^T_{k+1} \bold{R}^{-1}_{k+1} \bold{H}_{k+1} = \mathcal{A}_{1,k} +\alpha^{-1} h_{k+1} h_{k+1}^T$.
The result follows by applying Theorem~\ref{th:eig_bounds}.
\end{proof}

We formulate spectral bounds for $\mathcal{A}_1$ that depend on the largest and smallest eigenvalues of $\bold{D}$ and $\bold{R}$, and the largest and smallest singular values of $(\bL^T \  \bold{H}^T )$.

\begin{theorem}\label{th:1x1eig}
The eigenvalues of $\mathcal{A}_1$ lie in the interval
\begin{equation*}
I_+ = \left[\theta^2_{min} / \tau_{max}, \theta^2_{max}  / \tau_{min} \right],
\end{equation*}
where $\theta_i$ and $\tau_i$ are defined in Table~\ref{tab:notation_eigv_sv}, and \eqref{eq:tau_min} and \eqref{eq:tau_max}.
\end{theorem}

\begin{proof}
Assume that $\bold{u} \in \mathbb{R}^{(N+1)n}$ is an eigenvector of $\mathcal{A}_1$. Then the eigenvalue equation premultiplied by $\bold{u}^T$ can be written as
\begin{equation*}
\chi ||\bold{u}||^2 = \bold{u}^T (\bL^T \  \bold{H}^T ) \left( \begin{array}{cc}
\bold{D}^{-1}  & \bold{0} \\ 
 \bold{0} &  \bold{R}^{-1}
\end{array} \right)  \left( \begin{array}{c}
\bL \\
\bold{H}\end{array} \right) \bold{u},
\end{equation*}
where $\chi$ is an eigenvalue of $\mathcal{A}_1$. The smallest and largest eigenvalues of $ \left( \begin{array}{cc}
\bold{D}^{-1}  & \bold{0} \\ 
 \bold{0} &  \bold{R}^{-1}
\end{array} \right) $
are $\tau_{max}^{-1}$ and $\tau_{min}^{-1}$, respectively. 
The bounds follow from the following inequalities that are obtained using \eqref{energyLH}:
\begin{gather*}
\chi ||\bold{u}||^2  \geq \tau_{max}^{-1}  \bold{u}^T (\bL^T \  \bold{H}^T ) \left( \begin{array}{c}
\bL \\
\bold{H}\end{array} \right) \bold{u}  \geq \tau_{max}^{-1}  \theta^2_{min}  ||\bold{u}||^2, \\
\chi ||\bold{u}||^2  \leq \tau_{min}^{-1} \bold{u}^T (\bL^T \  \bold{H}^T ) \left( \begin{array}{c}
\bL \\
\bold{H}\end{array} \right) \bold{u}  \leq \tau_{min}^{-1} \theta^2_{max}  ||\bold{u}||^2.
\end{gather*}

\end{proof}

The following corollary explains how the upper bound for the eigenvalues of $\mathcal{A}_1$ changes with the addition of new observations. This result that applies for $\mathcal{A}_1$ with a general $\bold{R}$ is consistent with Theorem~\ref{th:eig1x1change} that considers $\mathcal{A}_1$ with a  diagonal $\bold{R}$.

\begin{cor}\label{cor:1block_upper}
The upper bound in Theorem \ref{th:1x1eig} moves away from zero or is unchanged when new observations are added.
\end{cor}
\begin{proof}
If $\tau_{min}= \rho_{min}$, $\tau_{min}$ decreases by Lemma \ref{comment:r_eigenvalues}. Otherwise $\tau_{min}$ does not change. The result follows by applying Lemma \ref{comment:theta_change} to determine the change to $\theta_{max}$. 
\end{proof}

It is unclear how the lower bound in Theorem \ref{th:1x1eig} changes with respect to the number of observations, because both the numerator and denominator grow or stay unchanged by Lemmas \ref{comment:theta_change} and \ref{comment:r_eigenvalues}, respectively.

\subsection{Alternative bounds}
\begin{sloppypar}
Alternative eigenvalue bounds for symmetric saddle point matrices have been formulated by Axelsson and Neytcheva \cite{Axelsson2006}. These depend on the eigenvalues of the matrices $\bL^T \bold{D}^{-1} \bL$, $\bold{R}$, $\bold{D}$ and $\mathcal{A}_1$, and ${\xi=max\{ | \lambda_i(\mathcal{A}_1^{-1/2} \bL^T \bold{D}^{-1} \bL \mathcal{A}_1^{-1/2})|, i=1, \dots, (N+1)n \}}$. 
\end{sloppypar}

\begin{theorem}[From Theorem 1 (c) of Axelsson and Neytcheva \cite{Axelsson2006}]\label{th:3x3_eig_bounds_AN2006}
The negative eigenvalues of $\mathcal{A}_3$ lie in the interval
\begin{equation*}
I_-  =  \left[ \frac{1}{2} \left(\tau_{max} - \sqrt{\tau_{max}^2 +4 \tau_{max} \lambda_{max}(\mathcal{A}_1)}\right), \frac{1}{2} \left(\tau_{min} - \sqrt{\tau_{min}^2 +4 \tau_{min} \lambda_{min}(\mathcal{A}_1)} \right) \right] 
\end{equation*}
and the positive ones lie in the interval
\begin{equation*}
I_+ =  \left[ \tau_{min}, \frac{1}{2} \left(\tau_{max} + \sqrt{\tau_{max}^2 +4 \tau_{max} \lambda_{max}(\mathcal{A}_1)}\right) \right].
\end{equation*}
\end{theorem}

Note that the lower bound for the positive eigenvalues in Theorem \ref{th:3x3_eig_bounds_AN2006} is the same as in Theorem \ref{the:3x3eig}. 

\begin{theorem}[From Theorem 1 (a) and (b) of Axelsson and Neytcheva \cite{Axelsson2006}]\label{th:2x2_eig_bounds_AN2006}
The negative eigenvalues of $\mathcal{A}_2$ lie in the interval
\begin{equation*}
I_- =  \left[ -\lambda_{max}(\mathcal{A}_1), \frac{-\lambda_{min}(\mathcal{A}_1)}{1+\frac{\xi \lambda_{min}(\mathcal{A}_1)}{\psi_{min}}}  \right], 
\end{equation*}
and the positive ones lie in the interval
\begin{equation}\label{eq:AN_posbound_2x2}
I_+ =  \left[ \psi_{min}, \frac{1}{2} \left( \psi_{max} + \sqrt{\psi_{max}^2 +4\psi_{max} \lambda_{max} (\bL^T \bold{D}^{-1} \bL)} \right) \right].
\end{equation}
\end{theorem}

We observe that the bound \eqref{eq:AN_posbound_2x2} for the positive eigenvalues, unlike our bound in Theorem~\ref{2blockintervals}, is independent of the number of observations. Also, in practical applications it may not be possible to compute the upper bound for the negative eigenvalues because of the $\xi$ term.

\section{Numerical Experiments}\label{sec:numerics}
\subsection{System setup}
We present results of numerical experiments using the Lorenz 96 model \cite{Lorenz96}, where the state of the system at time $t_i$ is $x_i = (X^1_i, X^2_i, \dots, X^n_i)^T$ and the evolution of $x_i$ components $X^j,\ j \in \{1,2,\dots,n\}$, is governed by a set of $n$ coupled ODEs:
\begin{equation*}\label{eq:lorenz96}
\frac{dX^j}{dt} = -X^{j-2} X^{j-1} + X^{j-1} X^{j+1} - X^j + F,
\end{equation*}
where $X^{-1} = X^{n-1}, X^0 = X^n$ and $X^{n+1} = X^1$. This model is continuous in time and discrete in space. We assume that $X^1, X^2 \dots, X^n$ are equally spaced on a periodic domain of length one and take the space increment to be $\Delta X = 1/n$. We require the linearisation of this model $M_i^{(l)}$, $i \in \{0,\dots,N-1\}$ to define $\mathcal{A}_3$, $\mathcal{A}_2$ and $\mathcal{A}_1$. In our experiments, we set $n=40$ and $F=8$, since the system shows chaotic behaviour with the latter value. The equations are integrated using a fourth order Runge-Kutta scheme \cite{Butcher87}. The time step is set to $\Delta t = 2.5 \times 10^{-2}$ and the system is run for $N=15$ time steps.

The assimilation system is set up for so-called identical twin experiments, by which synthetic data are generated using the same model as is used in the assimilation. We generate a reference, or "true", model trajectory $\bold{x}^t$ by running the Lorenz 96 model over the time window from prescribed initial conditions and with prescribed Gaussian model errors $\eta_i$. An initial background state $x^b$ and observations $y_i$ at each time $t_i$ are then generated by adding Gaussian noise to $\bold{x}^t$. Assimilation experiments are run using this background state and observations, assuming that the true state is unknown. The error covariance matrices that are used to generate the model error in $\bold{x}^t$ and the observation error in $y_i$ are also used for the assimilation, i.e. in the $3 \times 3$ block, $2 \times 2$ block and $1 \times 1$ block matrices. These error covariance matrices do not change over time. The observation error covariance matrix is $R_i=\sigma_o^2 I_{p_i}$, where $p_i$ is the number of observations at time $t_i$, (diagonal $R_i$ is a common choice in data assimilation experiments\cite{Freitag2018, Gratton2018}) and the model error covariance matrix is equal to the background error covariance matrix $Q_i=B=\sigma_b^2 C_b$, where $C_b$ is a Second-Order Auto-Regressive correlation matrix \cite{Daley91} with correlation length scale $1.5 \times 10^{-2}$. We have also performed numerical experiments with $Q_i = \sigma_q^2 C_q \neq B$, where $C_q$ is a Laplacian correlation matrix \cite{Johnson2005}, and $\sigma_q$ and $\sigma_b$ vary by a factor of two. We observed similar results to those presented here. In our experiments, the parameters are chosen so that the observations are close to the real values of the variables, and the background and the model errors are low, in particular, we set $\sigma_o=10^{-1}$, which is about $5$\% of the mean of the values in $\bold{x}^t$, and $\sigma_b = 5\times 10^{-2}$. $y_i$ consists of direct observations of the variables $X^j,\ j \in \{1,2,\dots,n\}$ at time $t_i$, hence the observation operator $\mathcal{H}_i$ is linear.

All computations are performed using Matlab R2016b. In particular, the eigenvalues are computed using the Matlab function \textit{eig}. If only extreme eigenvalues are needed, \textit{eigs} is used, and the extreme singular values are given by \textit{svds}.

\subsection{Eigenvalue bounds}
We present numerically calculated eigenvalue bounds and eigenvalues of $\mathcal{A}_3$, $\mathcal{A}_2$ and $\mathcal{A}_1$ and illustrate their change with the number of observations and the quality of the spectral estimates, presented in Section~\ref{sec:eigenvalues_theory}. 
We consider the following observation networks that have different numbers of observations ($p = \sum_{i=0}^N p_i$):
\begin{enumerate}[label=\alph*)] 
\item $1$ observation at the final time $t_{15}$,\label{obs_netw:1obs} 
\item $20$ observations, observing every  $8th$ model variable at every $4th$ time step (at times $t_{3},t_{7},t_{11},t_{15}$),  
\item $80$ observations, observing every $4th$ model variable at every $2nd$ time step (at times $t_{1},t_{3},t_{5},t_{7},t_{9},t_{11},t_{13},t_{15}$), \label{obs_netw:80obs}
\item $160$ observations, observing every $2nd$ model variable at every $2nd$ time step (at the same times as in observation network~\ref{obs_netw:80obs}),\label{obs_netw:160obs} 
\item $320$ observations, observing every $2nd$ model variable at every time step,  \label{obs_netw:320obs}
\item $640$ observations, fully observed system. \label{obs_netw:640obs}
\end{enumerate}

\begin{figure}[h!]
\begin{subfigure}[b]{0.5\linewidth}
  \centering
 \includegraphics[width=\linewidth]{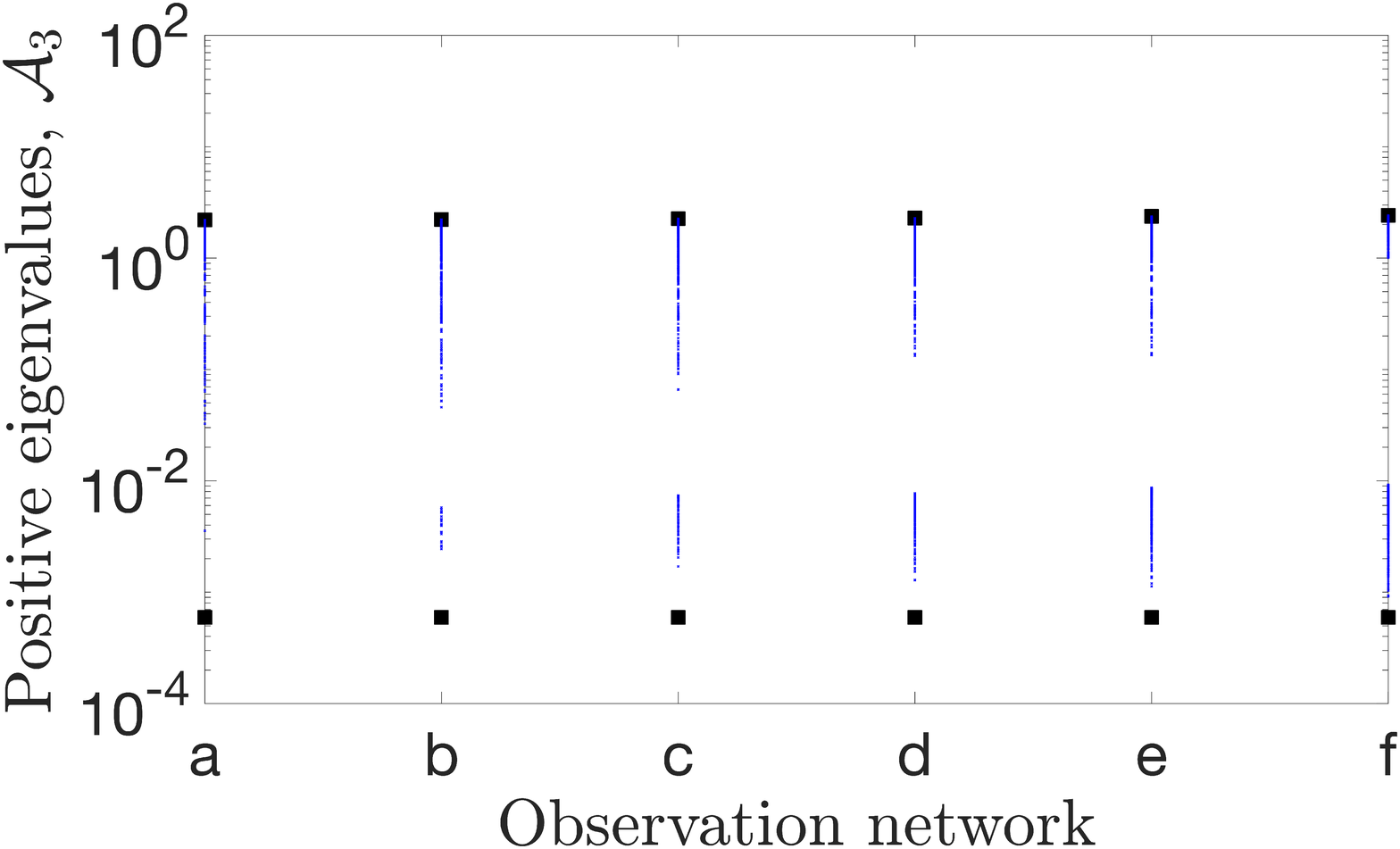}
   \caption{\label{fig:Figure1_i}}
\end{subfigure}
\begin{subfigure}[b]{0.5\linewidth}
  \centering
 \includegraphics[width=\linewidth]{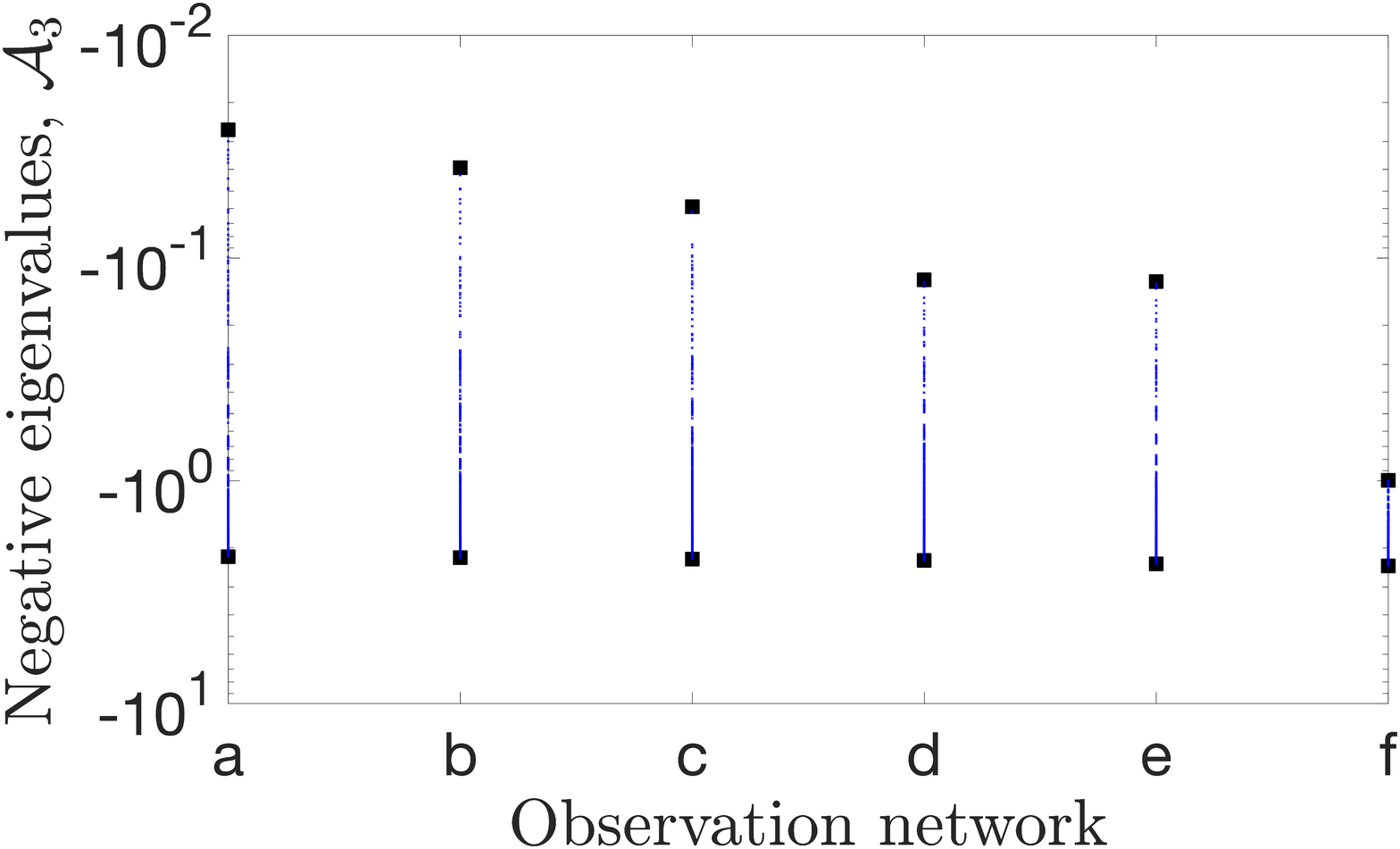}
     \caption{\label{fig:Figure1_ii}}
\end{subfigure}\\[1ex]
\begin{subfigure}[b]{0.5\linewidth}
  \centering
 \includegraphics[width=\linewidth]{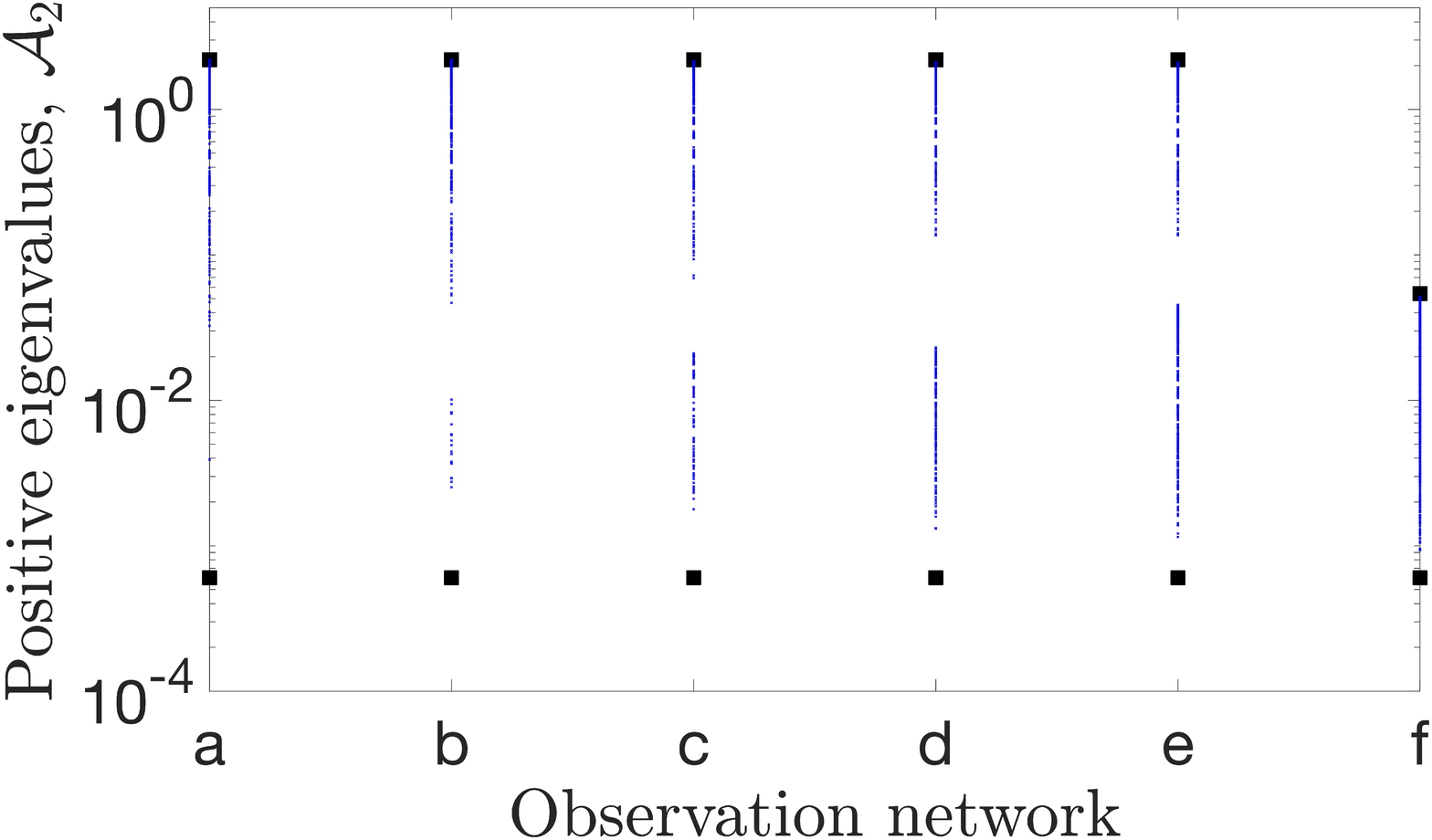}
     \caption{\label{fig:Figure1_iii}}
\end{subfigure}
\begin{subfigure}[b]{0.5\linewidth}
  \centering
 \includegraphics[width=\linewidth]{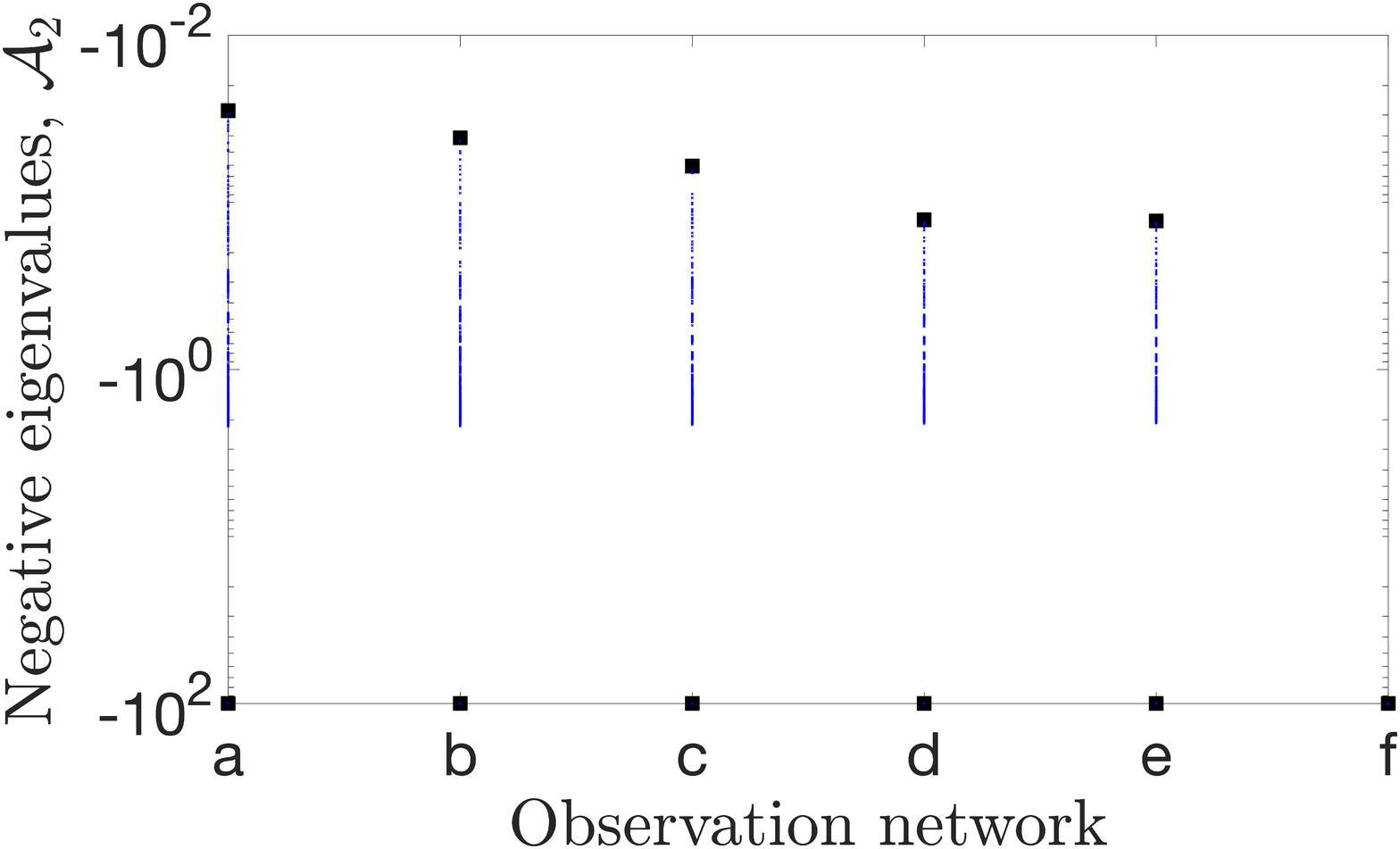}
     \caption{ \label{fig:Figure1_iv}}
\end{subfigure}\\[1ex]
\begin{subfigure}[b]{0.5\linewidth}
  \centering
 \includegraphics[width=\linewidth]{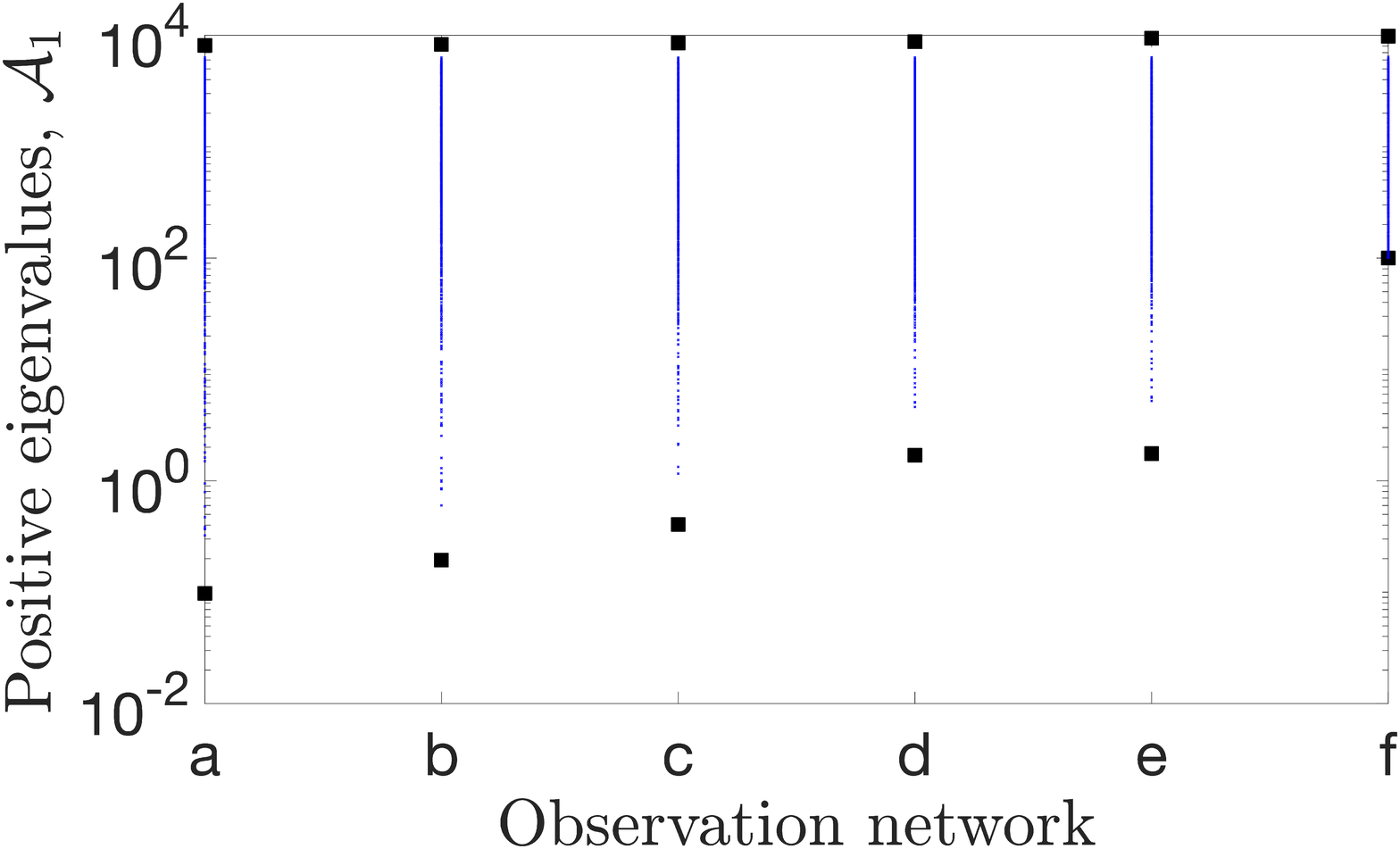}
     \caption{\label{fig:Figure1_v}}
\end{subfigure}\\[1ex]
\caption{Semi-logarithmic plots of the positive and negative eigenvalues of the matrices $\mathcal{A}_3$ (\subref{fig:Figure1_i} and \subref{fig:Figure1_ii}) and $\mathcal{A}_2$ (\subref{fig:Figure1_iii} and \subref{fig:Figure1_iv}), and the positive eigenvalues of $\mathcal{A}_1$ in \subref{fig:Figure1_v} for the different observation networks (a-f). Eigenvalues are denoted with merged blue dots. The filled black squares mark the bounds for eigenvalues of $\mathcal{A}_3$ in Theorem \ref{the:3x3eig}, $\mathcal{A}_2$ in Theorem \ref{2blockintervals}, and $\mathcal{A}_1$ in Theorem \ref{th:1x1eig}. Note that the smallest negative eigenvalues of $\mathcal{A}_2$ coincide with the bounds.}
\label{fig:Figure1}
\end{figure}

In Figure~\ref{fig:Figure1}, we plot the eigenvalues of the matrices $\mathcal{A}_3$, $\mathcal{A}_2$, and $\mathcal{A}_1$ together with the bounds from Theorems~\ref{the:3x3eig}, \ref{2blockintervals}, and \ref{th:1x1eig}, respectively, for each of the observation networks a-f. 
In these experiments, as expected from Theorem~\ref{th:3x3eig_change}, as the number of observations increases, the smallest and largest negative and the largest positive eigenvalues of $\mathcal{A}_3$ move away from zero and the smallest positive eigenvalue approaches zero. Also, as determined in Corollary~\ref{cor:3block_pos_bounds}, the upper bound for the positive eigenvalues of $\mathcal{A}_3$ presented in Figure~\ref{fig:Figure1_i} grows and the lower bound stays the same (because the eigenvalues of $\bold{R}$ do not change) when more observations are added. The change is too small to observe in the plots, hence we report the extreme eigenvalues of $\mathcal{A}_3$ and the intervals from Theorem~\ref{the:3x3eig} for the networks \ref{obs_netw:1obs}, \ref{obs_netw:80obs}, \ref{obs_netw:320obs} and \ref{obs_netw:640obs} in Table~\ref{table:bounds_eigenvalues_3x3}. Moreover, the negative bounds for the eigenvalues of $\mathcal{A}_3$ in Figure~\ref{fig:Figure1_ii} move away from zero. This agrees with Corollary \ref{cor:3block_neg_bounds}, because here $\tau_{min}=\psi_{min}$. However, in this setting $\tau_{max}=\rho_{max}$ and the same Corollary cannot be used to explain the change to the upper bound. In general, the outer bounds (the largest positive and the smallest negative) for the eigenvalues of $\mathcal{A}_3$ are tight and the inner bounds (the smallest positive and the largest negative) get tighter as the number of observations increases.
  
The positive eigenvalues of $\mathcal{A}_2$ displayed in Figure~\ref{fig:Figure1_iii} approach zero as observations are added, whereas the negative eigenvalues in Figure~\ref{fig:Figure1_iv} move away from it. This is consistent with Theorem~\ref{th:eigenvalues_2x2}, which holds for this experiment because we have chosen diagonal $\bold{R}$. The lower bounds for the positive and negative eigenvalues of $\mathcal{A}_2$ stay the same when the observation network is changed. In these bounds only $\nu_{max}$ (the largest eigenvalue of $\bold{H}^T \bold{R}^{-1} \bold{H}$) depends on the observations. In our experiments, $\nu_{max}$ does not change because of our choice of $\bold{H}$ and $\bold{R}$. The constant negative lower bound is consistent with Corollary~\ref{cor:2block_neg_lower_diag_R}. The numerical values of the intervals from Theorem~\ref{2blockintervals} and of the extreme eigenvalues of $\mathcal{A}_2$ for the networks \ref{obs_netw:1obs}, \ref{obs_netw:80obs}, \ref{obs_netw:320obs} and \ref{obs_netw:640obs} are presented in Table~\ref{table:bounds_eigenvalues_2x2}. The upper positive bound moves towards zero when the system becomes fully observed and is constant for the other networks, because the smallest eigenvalue $\nu_{min}$ of $\bold{H}^T \bold{R}^{-1} \bold{H}$ is non zero only for the fully observed system. The negative upper bound for the spectrum of $\mathcal{A}_2$ is given by $\beta_1$ in \eqref{2bnegupper1} for the fully observed system and $\beta_3$ in \eqref{2bnegupper3} otherwise, and moves away from zero, in agreement with Corollary~\ref{cor:2block_neg_upper}. Notice that the eigenvalue bounds are tight. Also, the numerical results confirm the statement of Corollary \ref{corollary:pos_eigen_bounds} that the interval for the positive eigenvalues of $\mathcal{A}_3$ contains the bounds for positive eigenvalues of $\mathcal{A}_2$.

Note that $\mathcal{A}_2$ has $p$ distinct eigenvalues that coincide with the negative lower bound in the plots. The distinct eigenvalues are explained by the bounds for individual eigenvalues in Corollary~\ref{cor:distinct_eig} in Appendix~\ref{ap:indiv_bounds_2x2}, because in our experiments $\bold{H}^T  \bold{R}^{-1} \bold{H}$ has eigenvalues that are equal to $\sigma_o^{-2}=10^{2}$ and the largest singular value $\sigma_{max}$ of $\bold{L}$ is less than $10$. Hence, there are $p$ eigenvalues of $\mathcal{A}_2$ in the interval $[-110,-90]$ and $(N+1)n-p$ eigenvalues no further than 10 from zero.

The eigenvalues of $\mathcal{A}_1$ and their bounds presented in Figure~\ref{fig:Figure1_v} move away from zero when more observations are used. This is as expected, because Theorem~\ref{th:eig1x1change} holds for our choice of diagonal $\bold{R}$. The variation of the bounds is explained by the fact that with our choice of $\bold{R}$ values of $\tau_{min}$ and $\tau_{max}$ do not change, and $\theta_{min}$ and $\theta_{max}$ grow. Such behaviour of the upper bound agrees with Corollary \ref{cor:1block_upper}. However, as can be seen in Table~\ref{table:bounds_eigenvalues_1x1} the upper value of the intervals in Theorem~\ref{th:1x1eig} are too pessimistic.

Better eigenvalue clustering away from zero when more observations are used can speed up the convergence of iterative solvers when solving the $1 \times 1$ block formulation. However, nothing definite can be said about the $3 \times 3$ block and $2 \times 2$ block formulations: the negative eigenvalues become more clustered, but the smallest positive eigenvalues approach zero when new observations are introduced.

\begin{table}
\centering
 \begin{tabular}{|c| c| c|c|c|} 
 \hline
 \rowcolor{lightgray}  O.n.  &  $I_-$  &  \cellcolor{lightgray}  Eigenvalues &  $I_+$ &  \cellcolor{lightgray}  Eigenvalues \\ 
 \hline
 \cellcolor{lightgray} \ref{obs_netw:1obs}  &  $ [-2.193,-2.66\times 10^{-2}]$   &  \cellcolor{lightgray} $[-2.192, -2.99\times 10^{-2}]$  & $[5.93\times 10^{-4},2.198]$ &  \cellcolor{lightgray}$[3.56\times 10^{-3}, 2.195]$
\\
 \hline
 \cellcolor{lightgray}  \ref{obs_netw:80obs}  &  $[-2.249,  -5.88\times 10^{-2}]$ &  \cellcolor{lightgray} $[ -2.247, -6.18\times 10^{-2}]$  &  $[5.93\times 10^{-4}, 2.254]$ &  \cellcolor{lightgray} $[1.70\times 10^{-3},  2.251]$\\
 \hline
 \cellcolor{lightgray} \ref{obs_netw:320obs}  & $[-2.360,-1.28\times 10^{-1}]$   &  \cellcolor{lightgray} $[-2.358,-1.31\times 10^{-1}]$ & $[5.93\times 10^{-4},2.365]$  &  \cellcolor{lightgray}  $[1.13\times 10^{-3},2.362]$ \\
 \hline
 \cellcolor{lightgray}  \ref{obs_netw:640obs} & $[-2.410,-9.96\times 10^{-1}]$ &  \cellcolor{lightgray} $[-2.408,-9.96\times 10^{-1}]$ &   $[5.93\times 10^{-4}, 2.416]$ & \cellcolor{lightgray}$[9.14\times 10^{-4}, 2.413]$ \\
 \hline
\end{tabular}
\caption{Computed spectral intervals and extreme eigenvalues of $\mathcal{A}_3$ from Theorem~\ref{the:3x3eig} for different observation networks (O.n.).}
\label{table:bounds_eigenvalues_3x3}
\end{table}

\begin{table}
\centering
 \begin{tabular}{|c| c| c|c|c|} 
 \hline
 \rowcolor{lightgray} O.n.   &  $I_-$  &  \cellcolor{lightgray}  Eigenvalues &  $I_+$ &  \cellcolor{lightgray}  Eigenvalues \\ 
 \hline
 \cellcolor{lightgray}  \ref{obs_netw:1obs}   &  $[-1.0005\times 10^2,-2.83\times 10^{-2}]$   &  \cellcolor{lightgray} $[-1.0001 \times 10^2,  -2.99\times 10^{-2}]$  & $[6.03\times 10^{-4},2.196] $&  \cellcolor{lightgray}$[3.91\times 10^{-3}, 2.195]$
\\
 \hline
 \cellcolor{lightgray}  \ref{obs_netw:80obs}  &  $[-1.0005\times 10^2,-6.07\times 10^{-2}]$  &  \cellcolor{lightgray} $[-1.0002\times 10^2,-6.50\times 10^{-2}]$  &  $[6.03\times 10^{-4},  2.196]$ &  \cellcolor{lightgray} $[1.78\times 10^{-3}, 2.148]$\\
 \hline
 \cellcolor{lightgray}  \ref{obs_netw:320obs} &  $[-1.0005\times 10^2,-1.29\times 10^{-1}]$     &  \cellcolor{lightgray} $[-1.0004\times 10^2,-1.33\times 10^{-1}]$  & $[6.03\times 10^{-4}, 2.196]$ &  \cellcolor{lightgray}  $[1.15\times 10^{-3}, 2.101]$ \\
 \hline
 \cellcolor{lightgray}  \ref{obs_netw:640obs} & $[-1.0005\times 10^2,-1.00\times 10^2]$ & \cellcolor{lightgray} $[-1.0005\times 10^2,-1.00\times 10^2]$ & $[6.03\times 10^{-4},5.42\times 10^{-2}]$ & \cellcolor{lightgray}$[9.35\times 10^{-4},5.15\times 10^{-2}]$\\
 \hline
\end{tabular}
\caption{Computed spectral intervals and extreme eigenvalues of $\mathcal{A}_2$ from Theorem~\ref{2blockintervals} for different observation networks (O.n.).}
\label{table:bounds_eigenvalues_2x2}
\end{table}  

\begin{table}
\centering
 \begin{tabular}{|c|c|c|} 
 \hline
 \rowcolor{lightgray} O.n.   &  $I_+$ &  \cellcolor{lightgray}  Eigenvalues \\ 
 \hline
 \cellcolor{lightgray} \ref{obs_netw:1obs}   &  $[9.72 \times 10^{-2}, 8.11 \times 10^3]$    &  \cellcolor{lightgray}$[3.23\times 10^{-1}, 6.30 \times 10^3]$
\\
 \hline
 \cellcolor{lightgray}  \ref{obs_netw:80obs} &  $[4.05 \times 10^{-1},8.53 \times 10^3]$  &  \cellcolor{lightgray} $[1.16, 6.32 \times 10^3]$\\
 \hline
 \cellcolor{lightgray}  \ref{obs_netw:320obs} &  $[1.75, 9.40  \times 10^3]$     &  \cellcolor{lightgray}  $[5.21, 6.35 \times 10^3]$ \\
 \hline
 \cellcolor{lightgray}  \ref{obs_netw:640obs} & $[1.00 \times 10^2, 9.80  \times 10^3]$  & \cellcolor{lightgray}$[1.00 \times 10^2,6.40 \times 10^3]$\\
 \hline
\end{tabular}
\caption{Computed spectral intervals and extreme eigenvalues of $\mathcal{A}_1$ from Theorem~\ref{th:1x1eig} with different observation networks (O.n.).}
\label{table:bounds_eigenvalues_1x1}
\end{table}  

We also calculate the alternative eigenvalue bounds given in Theorems~\ref{th:3x3_eig_bounds_AN2006} and \ref{th:2x2_eig_bounds_AN2006}. With the choice of parameters and observations considered in this section, the bounds given in these theorems are not as sharp as those in Theorems~\ref{the:3x3eig} and \ref{2blockintervals}. However, this is not always the case, as is illustrated in Tables~\ref{table:alternative_bounds3x3} and \ref{table:alternative_bounds2x2}. Here $\sigma_o=1.5$, $\sigma_b = 1$ and the observation network \ref{obs_netw:160obs} is used. 

\begin{table}
\centering
 \begin{tabular}{| c|c|c|} 
 \hline
\rowcolor{lightgray} Eigenvalues of $\mathcal{A}_3 $  &  \cellcolor{lightgray} Bounds from Th.~\ref{the:3x3eig}  &  Bounds from Th.~\ref{th:3x3_eig_bounds_AN2006} \\ 
 \hline
   $ [-1.93,-1.38\times 10^{-2}]$   &  \cellcolor{lightgray} $[-2.17,-5.83\times 10^{-3}]$  & $[-5.10,-1.33\times 10^{-2}]$ \\
   $[2.98\times 10^{-1},3.59]$ &  \cellcolor{lightgray} $[2.37\times 10^{-1},3.81]$  &  $[2.37\times 10^{-1}, 7.53]$ \\
 \hline
\end{tabular}
\caption{Computed spectral intervals and extreme eigenvalues of $\mathcal{A}_3$ from Theorems~\ref{the:3x3eig} and \ref{th:3x3_eig_bounds_AN2006} for observation network \ref{obs_netw:160obs} with $\sigma_o=1.5$ and $\sigma_b = 1$.}
\label{table:alternative_bounds3x3}
\end{table}

\begin{table}
\centering
 \begin{tabular}{|c|c|c|} 
 \hline
 \rowcolor{lightgray}  Eigenvalues of  $\mathcal{A}_2 $ &  \cellcolor{lightgray} Bounds from Th.~ \ref{2blockintervals}  &  Bounds from Th.~\ref{th:2x2_eig_bounds_AN2006} \\ 
 \hline
 $ [-1.97,-1.39\times 10^{-2}]$   &  \cellcolor{lightgray} $[-2.33,-5.83\times 10^{-3}]$  & $[-15.79, -1.33\times 10^{-2}]$ \\
 $[3.00\times 10^{-1},3.51]$ &  \cellcolor{lightgray} $[2.38\times 10^{-1},3.74]$  &  $[2.37\times 10^{-1},7.51]$ \\
 \hline
\end{tabular}
\caption{Computed spectral intervals and extreme eigenvalues of $\mathcal{A}_2$ from Theorems~\ref{2blockintervals} and \ref{th:2x2_eig_bounds_AN2006} for observation network \ref{obs_netw:160obs} with $\sigma_o=1.5$ and $\sigma_b = 1$.}
\label{table:alternative_bounds2x2}
\end{table}

\subsection{Solving the systems}
We solve the $3 \times 3$ block, $2 \times 2$ block, and $1 \times 1$ block systems with the coefficient matrices discussed in the previous subsection, and the right hand sides defined in \eqref{eq:saddle}, \eqref{2block_system}, and \eqref{eq:1block_system}, respectively. 
The saddle point systems are solved with MINRES and the symmetric positive definite systems are solved with CG. The relative residual at the $j$-th iteration of the iterative method is defined as $||\bold{r}_j||/||\bold{r}_0||$, where  $|| \cdot ||$  is the $L_2$ norm and $\bold{r}_j$ is the residual on iteration $j$. The iterative method terminates after $400$ iterations or when the relative residual reaches $10^{-4}$. The initial guess is taken to be the zero vector.

In Figure~\ref{fig:Figure2}, we plot the relative residuals. Note that the residual reaches $10^{-4}$ in the fully observed case (observation network~\ref{obs_netw:640obs}) when solving each of the systems and convergence is most rapid in this case. This is expected because of the clustering of the eigenvalues. The convergence rates are similar for networks d and e, which is consistent with Figure~\ref{fig:Figure1}. The convergence of MINRES for the observation network \ref{obs_netw:1obs} with a single observation is not explained by the spectra of  $\mathcal{A}_3$ and $\mathcal{A}_2$. However, the convergence in other cases agrees with our eigenvalue analysis.

\begin{figure}[h]
\begin{subfigure}[b]{0.5\linewidth}
  \centering
 \includegraphics[width=\linewidth]{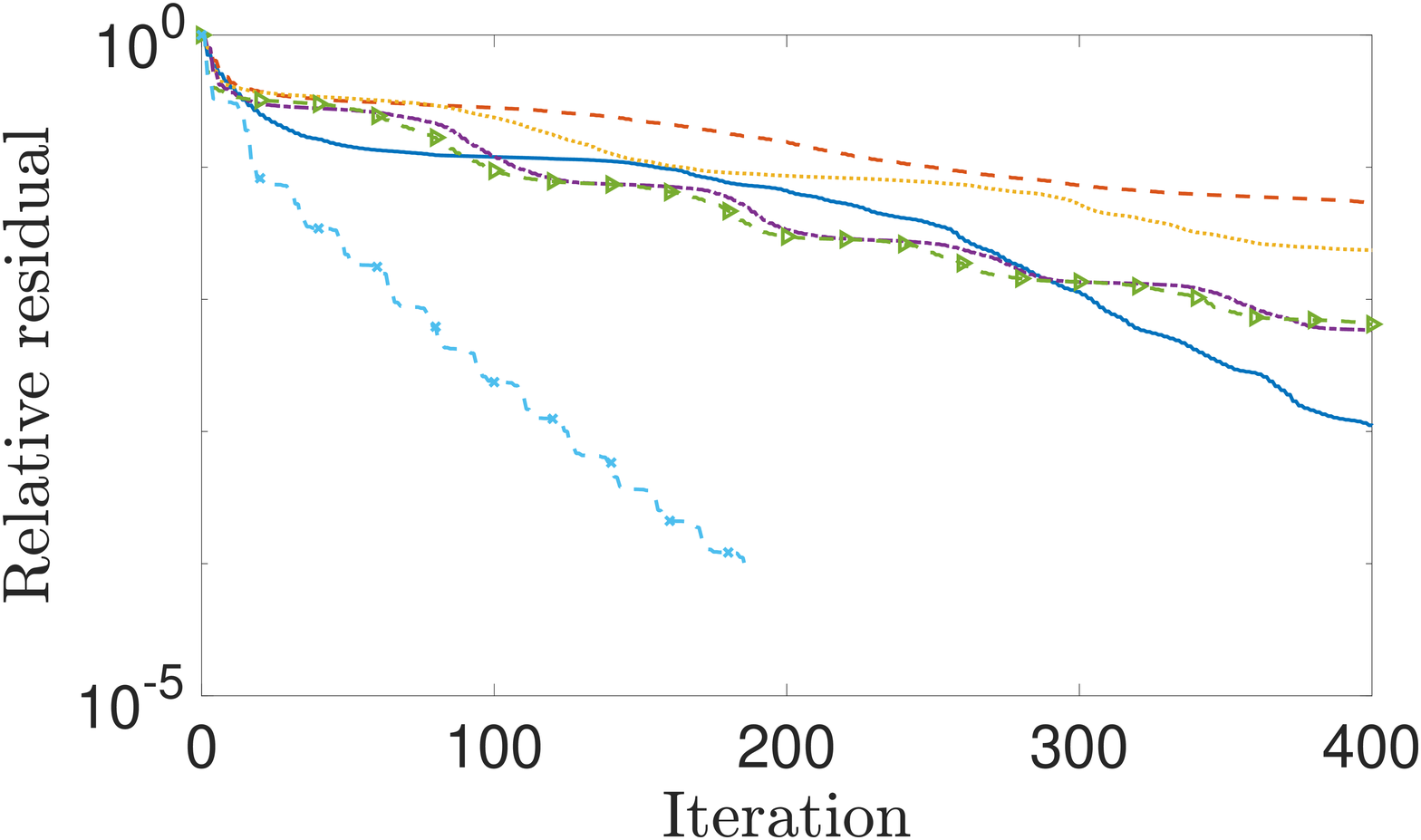}
   \caption{\label{fig:Figure2_i}}
\end{subfigure}
\begin{subfigure}[b]{0.5\linewidth}
  \centering
 \includegraphics[width=\linewidth]{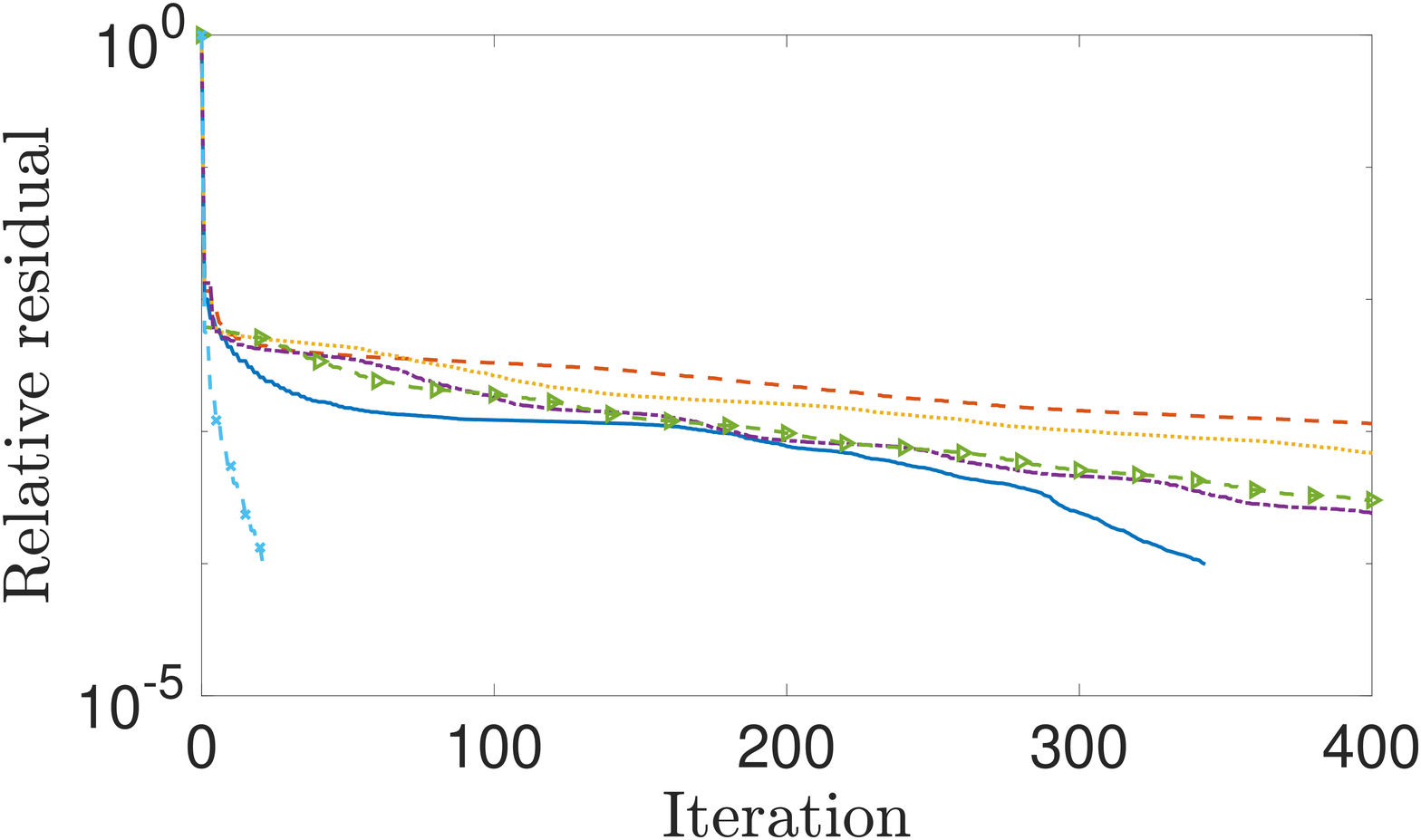}
     \caption{\label{fig:Figure2_ii}}
\end{subfigure}\\[1ex]
\begin{subfigure}[b]{0.5\linewidth}
  \centering
 \includegraphics[width=\linewidth]{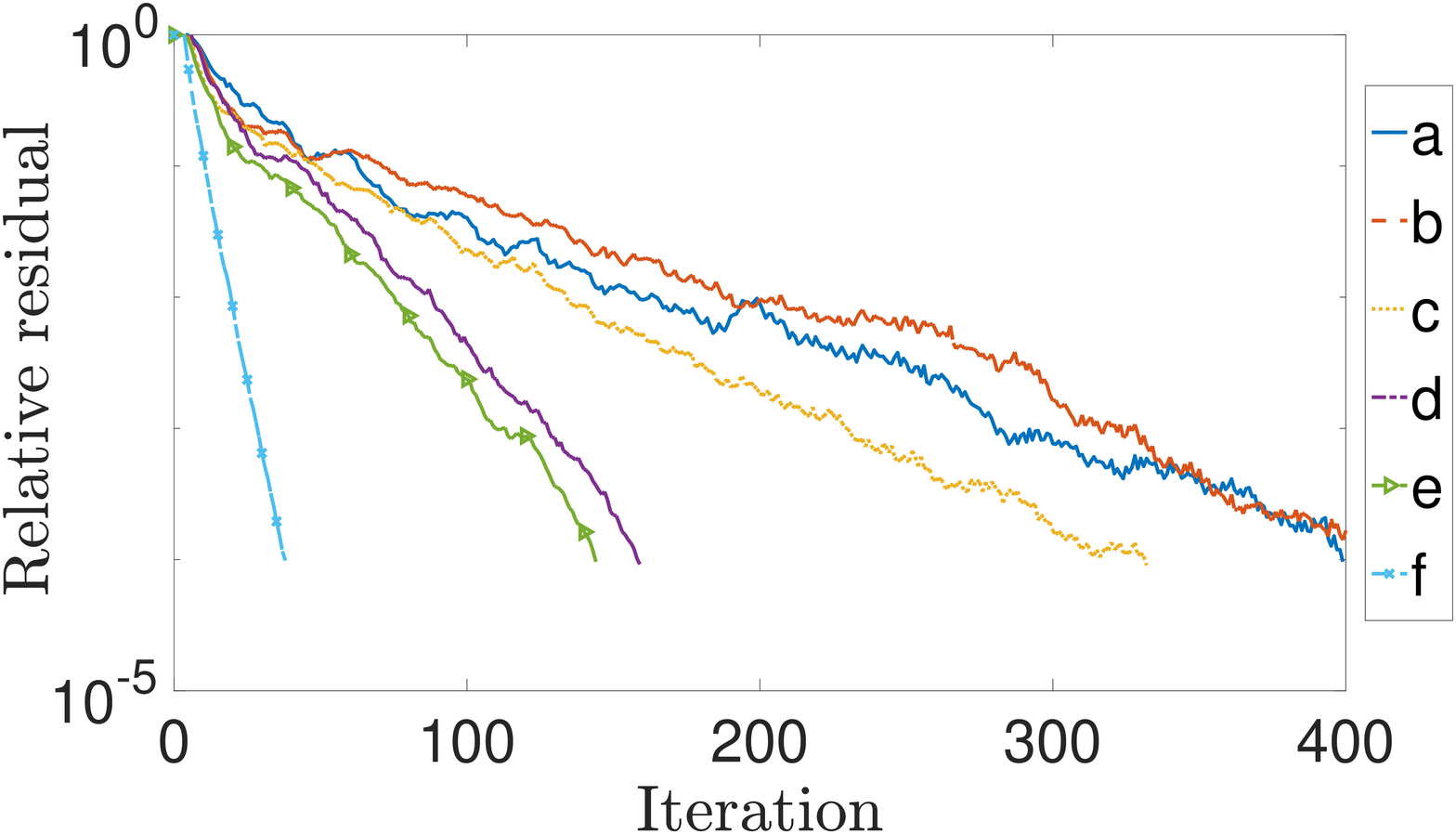}
     \caption{\label{fig:Figure2_iii}}
\end{subfigure}\\[1ex]
\caption{Semi-logarithmic plots of the relative residual of MINRES when solving the $3 \times 3$ block (I) and $2 \times 2$ block (II) systems, and the relative residual of CG when solving the $1 \times 1$ block (III) system for different observation networks (a-f). }
\label{fig:Figure2}
\end{figure}

\section{Conclusions}\label{sec:conclusions}
Weak constraint 4D-Var data assimilation requires the minimisation of a cost function in order to obtain an estimate of the state of a dynamical system. Its solution can be approximated by solving a series of linear systems. We have analysed three different formulations of these systems, namely the standard system with $1 \times 1$ block symmetric positive definite coefficient matrix $\mathcal{A}_1$, a new system with a $2 \times 2$ block saddle point coefficient matrix $\mathcal{A}_2$, and the version with $3 \times 3$ block saddle point coefficient matrix $\mathcal{A}_3$ of Fisher and G{\"u}rol\cite{Fisher17}. We have focused on the dependency of the coefficient matrices on the number of observations. 

We have found that the spectra of $\mathcal{A}_3$, $\mathcal{A}_2$ and $\mathcal{A}_1$ are sensitive to the number of observations and examined how they change when new observations are added. The results hold with any choice of the blocks in $\mathcal{A}_3$, whereas we can only make inference about the change of the spectra of $\mathcal{A}_2$ and $\mathcal{A}_1$ for uncorrelated observation errors (diagonal $\bold{R}$). We have shown that the negative eigenvalues of both $\mathcal{A}_3$ and $\mathcal{A}_2$ move away from zero or are unchanged when observations are added. The smallest and largest positive eigenvalues of $\mathcal{A}_2$, as well as the smallest positive eigenvalue of $\mathcal{A}_3$, approach zero or are unchanged, whereas the largest positive eigenvalue of $\mathcal{A}_3$ moves away from zero or is unchanged. The smallest and largest eigenvalues of $\mathcal{A}_1$ move away from zero or are unchanged. The extreme eigenvalues may cause convergence problems for Krylov subspace solvers, hence we may expect the small positive eigenvalues of $\mathcal{A}_2$ and $\mathcal{A}_3$ to cause these issues when new observations are added. We summarise these results together with the properties of the three systems in Table~\ref{table:all_formulations_summary}.

We have used the work of Rusten and Winther\cite{RustenWinther1992} to determine the bounds for the spectrum of $\mathcal{A}_3$ and derived novel bounds for the spectral intervals of the saddle point matrix $\mathcal{A}_2$ and the positive definite matrix $\mathcal{A}_1$. We have observed that the change to the intervals due to new observations is consistent with the change of the extreme eigenvalues of the matrices. Our numerical experiments agree with these findings. In general, the bounds for the saddle point matrices are tight whereas the upper bounds for the positive definite matrix are too pessimistic.

Our numerical experiments show slow convergence, particularly with a few observations, and the need for preconditioning is evident. Previous work on the $3 \times 3$ block saddle point system considered iteratively solving the problem when inexact constraint preconditioners of Bergamaschi et al.\cite{Bergamaschi07} are used (see, Fisher and G{\"u}rol\cite{Fisher17}, Freitag and Green\cite{Freitag2018}, Gratton et al.\cite{Gratton2018}). It was shown that such a preconditioning approach is not optimal and further research into effective preconditioning is still an open question. Preconditioning may transform the coefficient matrix into a non-normal one with GMRES as an iterative solver of choice. Although the spectrum of a non-normal matrix may not be enough to describe the convergence of GMRES \cite{Greenbaum1996}, Benzi et al. \cite{Benzi2005} claim that fast convergence often appears if the spectrum is clustered away from the origin. Hence, a better understanding of the spectrum of $\mathcal{A}_3$, $\mathcal{A}_2$ and $\mathcal{A}_1$ may help in finding a suitable preconditioner for these matrices. We suggest that including the information on observations coming from the observation error covariance matrix $\bold{R}$ and the linearised observation operator $\bold{H}$ could be beneficial for preconditioning, given that the spectra of all the considered matrices depend on the observations. A design of such preconditioners that are cheap to construct and apply is an interesting area for future research.

\begin{table}
\centering
 \begin{tabular}{|c| c| c|c|} 
 \hline
& $\mathcal{A}_3$ & $\mathcal{A}_2$   & $\mathcal{A}_1$  \\ 
 \hline
 Type & Symmetric  indefinite & Symmetric  indefinite & \begin{tabular}{c} Symmetric positive\\ definite \end{tabular} \\
  \hline
  Iterative solver &  MINRES/SYMMLQ & MINRES/SYMMLQ & CG\\ 
\hline
Order & $2(N+1)n+p$ & $2(N+1)n$ & $(N+1)n$  \\
  \hline
 $\bold{D}^{-1}$ needed & No & No & Yes \\
  \hline
 $\bold{R}^{-1}$ needed & No  & Yes  & Yes \\
  \hline
\begin{tabular}{c} Sequential matrix \\ products \end{tabular} & None & $\bold{H}^T \bold{R}^{-1} \bold{H}$ & \begin{tabular}{c}$\bL^T \bold{D}^{-1} \bL $, \\ $\bold{H}^T \bold{R}^{-1} \bold{H}$ \end{tabular} \\
 \hline
 \begin{tabular}{c}Eigenvalues that may move towards \\ zero with new observations  \end{tabular} & Smallest positive & Positive* & None*  \\
\hline
\begin{tabular}{c} Eigenvalues that may move away from \\  zero with new observations  \end{tabular} &\begin{tabular}{c} Largest positive,\\ negative  \end{tabular}  & Negative*  & All*  \\
\hline
\end{tabular}
\caption{A summary of the properties of the $3 \times 3$ block, $2 \times 2$ block, and $1 \times 1$ systems. * applies to systems with diagonal $\bold{R}$. }
\label{table:all_formulations_summary}
\end{table}

\section*{Acknowledgments}
We would like to kindly thank Dr. Adam El-Said for his code for the weak-constraint 4D-Var assimilation system. We are also grateful to two anonymous reviewers for their constructive comments that have led to improvements to the paper. This work does not have any conflicts of interest.

\section*{Funding information}
UK Engineering and Physical Sciences Research Council, Grant/Award Number: EP/L016613/1; European Research Council CUNDA project, Grant/Award Number: 694509; NERC National Centre for Earth Observation.

\appendix
\section{Bounds for individual eigenvalues of $\mathcal{A}_3$ and $\mathcal{A}_2$}\label{ap:indiv_bounds_2x2}
We derive bounds for the individual eigenvalues of $\mathcal{A}_3$ and $\mathcal{A}_2$ (Theorems~\ref{th:indiv_bounds3x3} and \ref{th:indiv_bounds2x2}, respectively). First, we state two theorems that are used in deriving these bounds. The notation of Table~\ref{tab:notation_eigv_sv} is used.

\begin{theorem}[See Theorem 3 in Silvester \cite{Silvester2000}]\label{th:block_det}
If $A = \left( \begin{array}{cc}
C & E \\
F & G\\
\end{array} \right) $ , $C, E, F, G \in \mathbb{R}^{n \times n}$, and $FG = GF$, then
\begin{center}
$ det(A) = det(CG - EF)$.
\end{center}
\end{theorem}

\begin{theorem}[Jordan-Wielandt Theorem, see Theorem 4.2 in Chapter 1 of Stewart and Sun \cite{Stewart1990}]\label{th:jordan-wielandt}
Let 
\begin{center}
$U^HAV = \left( \begin{array}{cc}
\Sigma & 0 \\
0 & 0\\
\end{array} \right)$, $\Sigma = diag(\sigma_1, \cdots, \sigma_n)$
\end{center} 
be the singular value decomposition of $A\in \mathbb{C^{m \times n}}$, $m \geq n$. Then the eigenvalues of the matrix 
\begin{center}
$C = \left( \begin{array}{cc}
0 & A \\
A^H & 0\\
\end{array} \right)$ 
\end{center}
are $\pm \sigma_1, \cdots, \pm \sigma_n$, corresponding to the eigenvectors $\left( \begin{array}{c} u_i \\ \pm v_i \end{array} \right)$, $i=1, \cdots, n$, where $u_i$ and $v_i$ are the $i$-th columns of $U$ and $V$, respectively. $C$ also has $m-n$ zero eigenvalues with eigenvectors $\left( \begin{array}{c} u_i \\ 0 \end{array} \right)$, $i=n+1, \cdots, m$.
\end{theorem}

\begin{theorem}\label{th:indiv_bounds3x3}
Let $\omega_i, \ i =1, \dots, (N+1)n+p$ be the $i$-th value in $\{\psi_k, \rho_j | k=1,\dots, (N+1)n, \ j=1,\dots, p \}$ (the set of eigenvalues of $\bold{D}$ and $\bold{R}$). Then the $k$-th eigenvalue of $\mathcal{A}_3$ is bounded by
\begin{align*}
\text{positive eigenvalues:}  \quad & \omega_{k} - \theta_{max} \leq \gamma_k \leq \omega_{k} + \theta_{max}, \quad & k=1, \dots, (N+1)n+p, \\ 
\text{negative eigenvalues: }\quad & - \theta_{max} \leq \gamma_{k+(N+1)n+p} < 0, \quad & k=1,\dots, (N+1)n.
\end{align*}
\end{theorem}

\begin{proof}
We can write $\mathcal{A}_3$ as a sum of two symmetric matrices:
\begin{equation*}
\mathcal{A}_3  = \left( \begin{array}{ccc}
\bold{D} & \bold{0} & \bold{L} \\
\bold{0} & \bold{R} & \bold{H} \\
\bold{L}^T & \bold{H}^T & \bold{0} \\
\end{array} \right) =
\left( \begin{array}{ccc}
\bold{D} & \bold{0} & \bold{0} \\
\bold{0} & \bold{R} & \bold{0} \\
\bold{0} & \bold{0} & \bold{0} \\
\end{array} \right) +
\left( \begin{array}{ccc}
\bold{0} & \bold{0} & \bold{L} \\
\bold{0} & \bold{0} & \bold{H} \\
\bold{L}^T & \bold{H}^T & \bold{0} \\
\end{array} \right) = \bold{S}^{3x3}_D + \bold{S}^{3x3}_L.
\end{equation*}

The spectrum of $\bold{S}^{3x3}_D$ is the union of the eigenvalues of $\bold{D}$, $\bold{R}$ and zeros. By Theorem~\ref{th:jordan-wielandt}, the eigenvalues $\lambda$ of the indefinite matrix $\bold{S}^{3x3}_L$ are the singular values of $(\bL^T \ \bold{H}^T)$ with plus and minus signs, thus $\lambda_{min} = -\theta_{max}$ and $\lambda_{max} =\theta_{max}$.

The result follows from applying Theorem~\ref{th:eig_bounds} to the matrices $\bold{S}^{3x3}_D$ and $\bold{S}^{3x3}_L$.
\end{proof}

\begin{theorem}\label{th:indiv_bounds2x2}
The eigenvalues of $\mathcal{A}_2$ are bounded by
\begin{align}
\text{positive eigenvalues:}  \quad & \psi_k - \sigma_{max} \leq \zeta_k \leq \psi_k + \sigma_{max}, \quad & k=1, \dots, (N+1)n. \nonumber \\
\text{negative eigenvalues: }  \quad & -\nu_{k} - \sigma_{max} \leq \zeta_{k+(N+1)n} \leq -\nu_{k} + \sigma_{max}, \quad & k=1, \dots, (N+1)n, \label{eq:bound_neg}
\end{align}
\end{theorem}

\begin{proof}
As in Theorem~\ref{th:indiv_bounds3x3}, we  express $\mathcal{A}_2$ as a sum of two symmetric matrices
\begin{equation*}
\mathcal{A}_2 =  \left( \begin{array}{cc}
\bold{D}  & \bold{0} \\
\bold{0} & -\bold{H}^T  \bold{R}^{-1} \bold{H}
\end{array} \right) +
\left( \begin{array}{cc}
\bold{0}  & \bL \\
\bL^T & \bold{0}
\end{array} \right) = \bold{S}^{2x2}_D + \bold{S}^{2x2}_L.
\end{equation*} 
The rest of the proof is analogous to that of Theorem~\ref{th:indiv_bounds3x3}.
\end{proof}

\begin{cor}\label{cor:distinct_eig}
If there are $p<(N+1)n$ observations, \eqref{eq:bound_neg} in Theorem \ref{th:indiv_bounds2x2} becomes
\begin{align*}
  - \sigma_{max} \leq & \zeta_{k+(N+1)n} \leq 0, \quad &  k=1, \dots, (N+1)n-p, \\
 -\nu_{k} - \sigma_{max} \leq &  \zeta_{k+2(N+1)n-p} <  -\nu_{k} + \sigma_{max}, \quad & k = 1, \dots, p.
\end{align*}
\end{cor}

\begin{proof}
The result follows from noticing that $-\bold{H}^T  \bold{R}^{-1} \bold{H}$ has $(N+1)n - p$ zero eigenvalues.
\end{proof}

\bibliographystyle{acm}
\bibliography{wileyNJD-VANCOUVER}
\end{document}